\documentclass[11pt]{amsart}
\usepackage{amscd,amssymb,latexsym,verbatim}
\usepackage{hyperref}
\usepackage{graphicx}
\usepackage[T1]{fontenc}
\usepackage{multirow}

\theoremstyle{plain}
\newtheorem{thm0}{Theorem}
\newtheorem*{thm1}{Theorem}
\newtheorem*{prop0}{Proposition~\ref{prop-suma}}
\newtheorem{thm}{Theorem}[section]
\newtheorem{lem}[thm]{Lemma}
\newtheorem{prop}[thm]{Proposition} 
\newtheorem{cor}[thm]{Corollary}

\theoremstyle{remark}
\newtheorem{rem}[thm]{Remark}

\theoremstyle{definition}
\newtheorem{dfn}[thm]{Definition} 
\newtheorem{example}[thm]{Example}

\newtheorem{step}{Step}

\newtheorem*{aspt}{Assumption}
\newtheorem{para}[thm]{}

\newcommand{\cD}{\mathcal{D}}
\newcommand{\cO}{\mathcal{O}}
\newcommand{\cF}{\mathcal{F}}
\newcommand{\cL}{\mathcal{L}}
\newcommand{\cV}{\Sigma}
\newcommand{\bc}{\mathbb{C}}
\newcommand{\CC}{\mathbb{C}}
\newcommand{\GG}{\mathbb{G}}
\newcommand{\bk}{\mathbb{K}}

\newcommand{\bn}{\mathbb{N}}
\newcommand{\bff}{\mathbb{F}}
\newcommand{\bq}{\mathbb{Q}}
\newcommand{\bz}{\mathbb{Z}}
\newcommand{\ZZ}{\mathbb{Z}}
\newcommand{\br}{\mathbb{R}}
\newcommand{\RR}{\mathbb{R}}
\newcommand{\bp}{\mathbb{P}}
\newcommand{\bt}{\mathbb{T}}

\newcommand{\one}{\mathbf{1}}

\DeclareMathOperator{\Hom}{Hom}
\DeclareMathOperator{\gl}{GL}
\DeclareMathOperator{\tors}{Tors}
\DeclareMathOperator{\sing}{Sing}
\DeclareMathOperator{\rk}{Rank}
\DeclareMathOperator{\lcm}{lcm}
\DeclareMathOperator{\Sh}{\operatorname{Shd}}
\DeclareMathOperator{\Char}{\operatorname{Char}}

\numberwithin{equation}{section}

\newcommand\enet[1]{\renewcommand\theenumi{#1}
\renewcommand\labelenumi{\theenumi}}

\title{Characteristic varieties of quasi-projective manifolds and orbifolds}

\author[E.~Artal]{Enrique Artal Bartolo}

\author[J.I.~Cogolludo]{Jos\'e Ignacio Cogolludo-Agust{\'\i}n}
\address{Departamento de Matem\'aticas, IUMA, Facultad de Ciencias\\
Universidad de Zaragoza\\
c/ Pedro Cerbuna 12\\
E-50009 Zaragoza SPAIN}
\email{artal@unizar.es,jicogo@unizar.es}

\author[D.~Matei]{Daniel Matei}
\address{Institute of Mathematics of the Romanian Academy\\ 
P.O. Box 1-764\\ RO-014700 Buch\-arest\\ Romania} 
\email{Daniel.Matei@imar.ro}

\thanks{Partially supported by the Spanish Ministry of Education
MTM2010-21740-C02-02. The third author is also partially supported
by grant CNCSIS PNII-IDEI 1188/2008 and FMI 53/10 (Gobierno de Arag{\'o}n).}
\begin{document}

\begin{abstract}
The present paper considers the structure of the space of characters of quasi-projective manifolds.
Such a space is stratified by the cohomology support loci of rank one local systems called characteristic 
varieties. The classical structure theorem of characteristic varieties is due to Arapura and it
exhibits the positive dimensional irreducible components as pull-backs obtained from morphisms onto
complex curves. 

In this paper a different approach is provided, using morphisms onto orbicurves, 
which accounts also for zero-dimensional components and gives more precise information on the positive
dimensional characteristic varieties. In the course of proving this orbifold version of 
Arapura's structure theorem, a gap in his proof is completed. As an illustration 
of the benefits of the orbifold approach, new obstructions for a group to be the fundamental group of 
a quasi-projective manifold are obtained.

%
\end{abstract}

\maketitle

\section*{Introduction}

The framework of this paper is the study of properties of fundamental groups of complements
of hypersurfaces in a projective space, or more generally, of smooth quasi-projective 
varieties. The approach we take is a classical one, namely to relate cohomological 
invariants of the variety (or its fundamental group) to its fibrations over a smooth curve, 
sometimes referred as pencils. This strong relationship has a long history, going back to 
Castelnuovo and de Franchis, see~\cite{cat:91}. The cohomological invariants we consider are 
the jumping loci of twisted cohomology of rank one local systems on the variety. The most 
general structure theorem for these loci was discovered by Arapura, who described them in 
terms of fibrations over curves. 

We propose here a different approach to obtain another structure theorem, 
where the base curve of the fibration is viewed as an orbifold. 
The language of orbifolds allows us to improve Arapura's description, and also to extract
finer quasi-projectivity obstructions. Our main goal is to prove the following result: 

\begin{thm0}\label{thmprin}
Let $X$ be a smooth quasi-projective variety and let $\cV_k(X)$ be the $k$-th characteristic 
variety of $X$. Let $V$ be an irreducible component of $\cV_k(X)$. 
Then one of the two following statements holds:
\begin{enumerate}
\enet{\rm(\arabic{enumi})}
\item\label{thmprin-orb} 
There exists an orbifold $C_\varphi$ supported by a smooth algebraic curve $C$, 
a surjective orbifold morphism $f:X\to C_\varphi$ and 
an irreducible component $W$ of $\cV_k(\pi_1^{\text{\rm orb}}(C_\varphi))$ 
such that $V=f^*(W)$.
\item\label{thmprin-tors} $V$ is an isolated torsion point not of type~\ref{thmprin-orb}.
\end{enumerate}
\end{thm0}

The characteristic varieties of a space depend only on its fundamental group
and can be seen as a generalization of the Alexander polynomial. They can be defined in terms of jumping
loci of the cohomology of local systems. These invariants have been extensively studied from different
perspectives. They are closely related with the Green-Lazarsfeld invariants~\cite{grlz:87} and with 
the Bieri-Neumann-Strebel invariants~\cite{bns} of groups and spaces.

In this context, the study of the geometry of smooth quasi-projective varieties in terms
of fibrations onto Riemann surfaces has proved to be very fruitful as its widespread use 
shows, cf. the following contributions by Siu, Serrano, Beauville, Catanese, Simpson, Bauer,  
and Arapura~\cite{siu, srn:90, be, cat:91, Si2, ara:97}. 

This paper originated from our attempt to understand Arapura's work. 
In~\cite{ara:97} the following result is stated in Theorem V.1.6.

\begin{thm1}[Arapura~\cite{ara:97}]
\label{thm-arapura}
Let $V$ be an irreducible component of $\cV_1(X)$. We then have
\begin{enumerate}
\enet{\rm(\arabic{enumi})}
\item\label{thm-arapura-1}
If $\dim V>0$, then there exists a surjective morphism $f:X\to C$, onto a smooth 
algebraic curve $C$, and a torsion character $\tau$ such that $V=\tau f^*(H^1(C;\bc^*))$.
\item If $\dim V=0$, then $V$ is unitary.
\end{enumerate}
\end{thm1}

This theorem is a consequence of Proposition V.1.4 from~\cite{ara:97}. 
However, the proof of this proposition given by Arapura is not complete.
The key technical tool used there is Timmerscheidt's spectral sequence degeneracy 
result~\cite[Theorem~5.1]{timm:87} for unitary local systems $\xi$ on $X$. 
Note that the relevant $E_1$-terms are not associated with the divisor $\cD$ 
compactifying $X=\bar{X}\setminus\cD$, but rather with the subdivisor $\cD^{\xi}$ of 
$\cD$ consisting of those components along which $\xi$ has non-trivial monodromy.
The starting assumption of Arapura's proof, that one can just deal with a local
system having nontrivial local monodromy about all components of the divisor 
$\cD$ after replacing $X$ by $X^{\xi}=\bar{X}\setminus \cD^{\xi}$, cannot in fact be made. 
Indeed, as it can be seen in Example~\ref{ex-notram} below, the resulting local 
system may no longer be in $\cV_1(X^{\xi})$. 
Characters that do not ramify along all components of $\cD$ have been considered
in the context of complements to projective hypersurfaces (cf.~\cite{acc:pcps})
and they seem to be essentially different from those ramifying everywhere as
can be seen in the Hodge-theoretical characterization provided 
in~\cite[Theorem~5.1]{acl-depth} where a character is not ramified everywhere
iff it is of weight two. Different techniques than those coming from cohomological
properties of compact manifolds are hence required in order to deal with this phenomenon.


Our first goal is to fill the gap in Arapura's proof. This is done in Section~\ref{sec-char}
in the language of orbifold morphisms. Proposition~\ref{prop-iso} ensures the validity of
Arapura's statement in~\cite[Proposition V.1.4]{ara:97}, and Theorem~\ref{thm-pullback} is 
an extension of that statement to non-torsion characters. 
We also point out that Theorem~\ref{thmprin} works for any characteristic variety and not only for~$\cV_1$.

The second goal is to describe, in terms of orbifold morphisms, the 
\emph{translated components} appearing in Arapura's work, see also~\cite{Di1} for another approach.
Several papers in this direction can be found in the literature. In particular, Theorem~\ref{thmprin} 
can be considered as the quasi-projective version of Delzant's Theorem~\cite{delzant}. 

Orbifolds have also been used for the study of fundamental groups of smooth algebraic varieties in 
recent works by Corlette-Simpson~\cite{cs:08} and Campana~\cite{camp:07,camp:10}. 
In ~\cite{delzant} Delzant proved the exact compact K\"ahler manifold analogue of Theorem~\ref{thmprin}, which was extended to
compact K\"ahler orbifolds in~\cite{camp:09}.

The third goal is to rule out the existence of non-torsion isolated unitary points in the characteristic 
varieties of quasi-projective manifolds. In the compact projective (resp. K\"ahler) case this was known 
from Simpson~\cite{Si1} (resp. Campana~\cite{camp:01}). In the non-compact case it was proved by 
Libgober~\cite{li:09} for quasi-projective manifolds $X$ with $b_1(\bar{X})=0$. 
Dimca in~\cite{Di1} pointed out that this fact may be deduced from Budur's work~\cite{Bu} for general quasi-projective 
varieties. This can be done with a big amount of work, however, we present here a more direct proof of this fact.

The fourth and final goal is to derive properties of the fundamental groups of quasi-projective 
varieties from Theorem~\ref{thmprin}, which can be used as effective quasiprojectivity
obstructions. As a matter of example, 
we present here one of the more striking consequences of Theorem~\ref{thmprin}.
\begin{prop0}
Let $G$ be a quasi-projective group, and let $V_1$ and $V_2$ be two distinct irreducible components 
of $\cV_k(G)$, resp. $\cV_\ell(G)$. If $\xi\in V_1\cap V_2$ is a torsion point, then~$\xi\in\cV_{k+\ell}(G)$.
\end{prop0}
This result was used in~\cite{ACM-artin} to rule 
out certain families of groups as fundamental groups of quasi-projective manifolds. 
We collect other properties in Propositions~\ref{prop-obs1} and~\ref{prop-obs2} 
In this way, we recover and extend some of the properties found by Dimca, Papadima, and 
Suciu~\cite{dps:08,DPS4,Di1,Di3,Di4,DM} using Arapura's Theorem.

This paper was originally planned with another goal in mind: to prove that only components of 
type~\ref{thmprin-orb} could exist. This was in part justified by the heuristic fact that all the computed 
examples were of type~\ref{thmprin-orb} (included the isolated points), but one can also give results in 
this direction in the rational surface case, where characters of orders 2, 3, 4, and 6 are of 
type~\ref{thmprin-orb} for divisors with rational singularities (see~\cite{acl-depth,CL-prep}).

However, the first two authors have recently found an example (see~\cite{AC-prep}) of a quasi-projective 
surface whose isolated points of $\cV_1$ are of type~\ref{thmprin-tors}. 

The paper is organized as follows. In~\S\ref{sec-prelim} notations for quasi-projective varieties are set
and some of their properties are discussed. The concept of characteristic variety is introduced in terms 
of Betti numbers of the $1$-cohomology with values in local systems of coefficients. Some ways to compute 
this cohomology are sketched and we applied them in~\S\ref{sec-orb} 
to compute characteristic varieties of orbifolds; these computations are probably known for specialists
but they do not appear explicitly in the literature. For quasi-projective varieties, the main tool of computation comes from 
Deligne's work~\cite{del:70}, which is recalled in~\S\ref{sec-deligne}. Also in this section, some technical
results are proved. They are mainly contained, one way or another, in the work of Arapura~\cite{ara:97}, 
Timmerscheidt~\cite{timm:87}, and Beauville~\cite{be}. The purpose of this is two-fold: on one hand to 
strengthen the use of the \emph{unitary-holomorphic} decompositions of a character~\cite{ara:97,be}; on the 
other hand to prepare the ground for a more precise analysis of the Deligne decomposition of the $1$-cohomology 
for local systems of coefficients, which will be considered in~\S\ref{sec-char}.  
In~\S\ref{sec-inverse}
the Deligne's decomposition into holomorphic
and anti-holomorphic parts is analyzed for Riemann surfaces as a main ingredient for
the general decomposition. In~\S\ref{sec-char} this analysis
is carried over to higher-dimensional varieties. The main results are stated and proved in section~\S\ref{sec-char}. 
The key ingredient is Theorem~\ref{thm-pullback} which allows us to prove Theorem~\ref{thmprin} for 
$\cV_k$ for any~$k$, since it states that any element of the twisted cohomology for a quasi-projective 
group (as long as the character is not of torsion type) is obtained as the pull-back of an element of the 
twisted cohomology of the orbifold. The strategy to prove Theorem~\ref{thm-pullback} is to reduce
it to the holomorphically pure cohomology classes. To that end, we first investigate in 
Proposition~\ref{prop-iso} the relation between the anti-holomorphic parts associated with a character 
and to its conjugate character. Proposition~\ref{prop-ara-beau} is a generalization of~\cite[Proposition~V.1.3]{ara:97} 
to the orbifold case. The results of~\S\ref{sec-deligne} are then used to apply Arapura's method to characters which are 
non-torsion, not just non-unitary. These improvements 
allow us to deal with any non-torsion unitary characters using Delzant's approach. The proof of Proposition~\ref{prop-caso2} 
uses Levitt's interpretation~\cite{lev:94} of \emph{exceptional classes} and Simpson's main result in~\cite{Si2}
can be applied. A more direct approach using Delzant's way could be done if the result in~\cite{Si2} were 
generalized to the quasi-projective case. In~\S\ref{sec-torsion}, some improvements of the main theorem are
discussed for torsion characters. They are not included in the main theorem because the hypotheses are 
rather technical. This section also includes a number of applications which follow from Theorem~\ref{thmprin}
and~\S\ref{sec-char}. Finally, some examples illustrating the properties are shown in~\S\ref{sec-exam}. 
More examples, applying these techniques can also be found in~\cite{ACM-artin}.

\section{Preliminaries}\label{sec-prelim}

Let $X$ be a smooth quasi-projective variety. Using standard Lefschetz-Zariski theory, see~\cite{hamm:83},
since the invariants we are interested in only depend on 
$G:=\pi_1(X)$, $X$ will be assumed to be either a complex curve (a Riemann surface) or a complex surface. 
In any case there exists a smooth compact complex surface (or complex curve) $\bar{X}$ such that 
$X=\bar{X}\setminus\cD$, where $\cD$ is a normal crossing divisor. If necessary, additional blow-ups might be 
performed in order to obtain a more suitable~$\bar{X}$, which will be clear from the context. 

Characteristic varieties are invariants of finitely presented groups $G$, and they can be computed using
any connected topological space $X$ (having the homotopy type of a finite CW-complex) such that 
$G=\pi_1(X,x_0)$, $x_0\in X$ as follows. Let us denote $H:=H_1(X;\bz)=G/G'$.
Note that the space of characters on $G$ is a complex torus
\begin{equation}
\label{eq-torus}
\bt_G:=\Hom(G,\bc^*)=\Hom(H,\bc^*)=H^1(X;\bc^*).
\end{equation}
Given $\xi\in\bt_G$, the following local system $\bc_\xi$ of coefficients over $X$ can be constructed.
Let $\pi_{\text{\rm ab}}:\tilde{X}_{\text{\rm ab}}\to X$ be the universal abelian covering of $X$. The group $H$ acts freely 
(on the right) on $\tilde{X}_{\text{\rm ab}}$ by the 
deck transformations of the covering. The local system of coefficients $\bc_\xi$ is defined as 
the locally constant sheaf associated with:
\[
\pi_\xi:\tilde{X}_{\text{\rm ab}}\times_{H}\bc\to X
\text{ where }
\tilde{X}_{\text{\rm ab}}\times_{H}\bc:=
\left(
\tilde{X}_{\text{\rm ab}}\times\bc
\right)\Big/
(x,t)\sim(x^h,\xi(h^{-1}) t).
\]

\begin{dfn}
\label{def-char-var}
The $k$-th~\emph{characteristic variety}~of $G$ is the subvariety 
of $\bt_G$, defined by:
\[
\cV_{k}(G):=\{ \xi \in \bt_G\: |\:\dim H^1(X,\bc_{\xi}) \ge k \}, 
\]
where $H^1(X,\bc_{\xi})$ is classically called the 
\emph{twisted cohomology of $X$ with coefficients in the local system $\xi\in \bt_G$}.

It is also customary to use $\cV_{k}(X)$ for $\cV_{k}(G)$ whenever $\pi_1(X)=G$.
\end{dfn}

\begin{para}\textbf{Topological construction of $H^1(X;\bc_\xi)$.}\label{rem-cell}
By duality, we will concentrate our attention in describing the simpler object $H_1(X;\bc_\xi)$.
Let us suppose that $X$ is a finite CW-complex. 
Then, $\tilde{X}_{\text{\rm ab}}$ also inherits a CW-complex structure. 
Since $H$ is the group of automorphisms of~$\pi_{\text{\rm ab}}$, $H$ acts freely on 
the set of cells of $\tilde{X}$, and thus the chain complex $C_*(\tilde{X}_{\text{\rm ab}};\bc)$ 
becomes a free $\Lambda$-module of finite rank, 
where $\Lambda:=\bc[H]$ is the group algebra of $H$. Given a character~$\xi\in \bt_G$, $\bc$ acquires a structure 
of $\Lambda$-module $\bc_\xi$ (obtained by the evaluation of $\xi$ on the elements of $H$). 
The \emph{twisted chain complex} $C_*(X;\bc)^\xi:=C_*(\tilde{X}_{\text{\rm ab}};\bc)\otimes_\Lambda\bc_\xi$ 
is, as a vector space, isomorphic to the finite dimensional complex space $C_*(X;\bc)$ with twisted differential. This construction implies 
that $\cV_k(G)$ are algebraic subvarieties of $\bt_G$ defined over~$\bq$.

Moreover, 
\begin{equation*}
\cV_k(G)\setminus \one=\Char_{k}(H^1(\tilde X_{\text{\rm ab}}))\setminus \one
=\Char_{k+1}(\tilde H^1(\tilde X_{\text{\rm ab}}))\setminus \one\
\text{(cf.\cite[\S1.2]{ji:mams} or \cite{eh:97})},
\end{equation*}
where, if $M$ is a finitely generated $\Lambda$-module then $\Char_k(M)$ is the algebraic variety associated
with the annihilator of the module $\bigwedge^k M$.
Finally, a presentation matrix for the 
module $\tilde H^1(C_*(X;\bc)^\xi)$ is given by evaluation of the Fox matrix of $G$
via the character~$\xi$. This matrix is obtained using Fox calculus~\cite{fox:60}
and it will be extensively used in~\S\ref{sec-orb}.
\end{para}

\begin{example}
\label{ex-cw}
Let $G:=\langle x_1,\dots,x_n\mid R_1,\dots,R_s\rangle$ be a presentation of $G$ and let
$K$ be the CW-complex associated with the presentation. This CW-complex has one
$0$-cell, say $P$, $n$~$1$-cells denoted by $x_1,\dots,x_n$ such that the $1$-skeleton
is a wedge of $n$~circles, and $s$~$2$-cells~$R_1,\dots,R_s$ such that their attachments to the $1$-skeleton
are determined by the corresponding words. Let $\xi:G\to\bc^*$ be a character and
let us denote $t_i:=\xi(x_i)$. Then the twisted differential of complex $C_*(K;\bc)^\xi$ are defined as follows:
\begin{itemize}
\item $\partial_1(x_i)=(t_i-1) P$;
\item The map $\partial_2$ is determined after a map $\varphi$ defined on the free group in $x_1,\dots,x_n$ which is 
defined inductively:
\begin{itemize}
\item The image of the empty word by $\varphi$ is $0$.
\item For a word $w$, we have:
\begin{itemize}
\item $\varphi(x_i w):=x_i+ t_i \varphi(w)$.
\item $\varphi(x_i^{-1} w):=-t_i^{-1} x_i+ t_i^{-1} \varphi(w)$.
\end{itemize}
\end{itemize}
\end{itemize}
Note that $\varphi(x_i^k w)=\frac{t_i^k-1}{t_i-1} x_i+t_i^k \varphi(w)$.
\end{example}

\begin{para}\textbf{Algebraic construction of $H^1(X;\bc_\xi)$.}\label{alg-const}
There is another way to compute the cohomology $H^1(X;\bc_\xi)=H^1(G;\bc_\xi)$, which can be seen as 
the quotient of \emph{cocycles} by \emph{coboundaries}. A cocycle is a map $\alpha:G\to\bc$ such that
$\alpha(g h)=\alpha(g)+\xi(g)\alpha(h)$. Coboundaries are generated by the mapping $g\mapsto\xi(g)-1$.
A cocycle defines a representation $G\to\gl(2;\bc)$, 
$g\mapsto \left(\begin{smallmatrix}
\xi(g)&\alpha(g)\\
0&1
\end{smallmatrix}\right)$.
Note that the coboundary representation is reducible.
\end{para}

\begin{rem}
\label{rem-ses-torus}
Let $r:=\rk H$ and let $\tors_G$ be the torsion subgroup of $H=G/G'$. 
Then $\bt_G$ is an abelian complex Lie group with $|\tors_G|$ connected components (each one isomorphic 
to $(\bc^*)^r$) satisfying the following exact sequence:
\begin{equation}
\label{eq-ses-torus}
1\to \bt_G^\one\to\bt_G\to \Hom(\tors_G,\CC^*)\to 1,
\end{equation}
where 
$\bt_G^\one:=\Hom(H/\tors_G,\CC^*)$
is the connected component containing the trivial character $\one$ which is isomorphic to $(\bc^*)^r$
via the choice of a basis of the lattice $H/\tors_G$.
Given $\rho\in \Hom(\tors_G,\CC^*)$, we will refer to the component of $\bt_G$ whose image is
$\rho$ as $\bt_G^\rho$. Since the exact sequence~(\ref{eq-ses-torus}) splits, the elements $\rho$ can be 
considered in $\bt_G$, and $\bt_G^\rho$ can also be thought of as the only connected component of $\bt_G$ 
passing through $\rho$.
\end{rem}

When $X$ is a quasi-projective (or K\"ahler) manifold, the work of Deligne~\cite{del:70}
gives a way to compute the twisted cohomology in terms of geometric properties.
Let us give some details about this computation. Fix a projective manifold $\bar{X}$ 
such that $\cD:=\bar{X}\setminus X$ is a normal crossing divisor.

\begin{dfn}
Let $D$ be an irreducible component of $\cD$ and let $\xi\in H^1(X;\bc^*)$. We say that
$\xi$ \emph{does not ramify}, or \emph{has trivial monodromy}, along $D$ if $\xi(\mu_D)=1$,
where $\mu_D$ is a meridian of $D$. Otherwise, we say that $\xi$ \emph{ramifies} along $D$.
\end{dfn}

\begin{rem}
By a \emph{meridian} of $D$, we simply mean the boundary of a sufficiently small disk intersecting
$D$ transversally and only at one point. This means nothing but the 
boundary of a fiber of the tubular neighborhood (seen as sub-bundle of the normal fibered bundle) on the 
smooth part of $D$. Since the tubular neighborhood of the smooth part of $D$ in $\cD$ is connected,
all meridians of $D$ are 
homologically equivalent.
\end{rem}

\begin{rem}\label{rem-notram}
Let $\cD^\xi$ be the subdivisor of $\cD$ formed by the irreducible components $D$ of $\cD$ such
that $\xi$ ramifies along $D$. Let $X^\xi:=\bar{X}\setminus\cD^\xi$ and $G^\xi:=\pi_1(X^\xi)$. Despite
the notation, $X^\xi$ and $G^\xi$ depend on $\bar{X}$. Note that 
$\xi$ naturally determines an element $\xi_0\in\bt_{G^\xi}=H^1(X^\xi;\bc^*)$.
It is clear that $\xi_0\in\cV_1(X^\xi)$ implies that $\xi\in\cV_1(X)$, but note that the converse is not 
true in general as the following example shows. This stresses a common as well as subtle misconception 
when trying to study characters on $X$. Therefore, one cannot assume, changing $X$ by $X^\xi$, that a 
character $\xi$ ramifies along all irreducible components of $\cD$.
\end{rem}

\begin{example}\label{ex-notram}
Let $X:=\bp^1\setminus\{0,1,\infty\}$. The group $G:=\pi_1(X)$ is free of rank~$2$ (generated, e.g., by 
meridians of $0$ and $\infty$). The torus $\bt_G$ is identified with $(\bc^*)^2$ via the images of those 
meridians. It is not hard to prove that $\cV_1(X)=\bt_G$, see Proposition~\ref{charvar-libre}. Let us 
consider a character $\xi$ defined by $(z,z^{-1})$, $z\in\bc\setminus\{0,1\}$. In this case, the divisor 
$\cD$ is the set $\{0,1,\infty\}$ but the divisor $\cD^\xi$ is  $\{0,\infty\}$ since $\xi$ does not ramify 
at~$1$. Note that $X^\xi=\bc^*$, whose fundamental group is Abelian. Since $\xi_0$ is a non-trivial character,
$\xi_0\notin\cV_1(X^\xi)$.
\end{example}

\section{Orbifold groups and characteristic varieties}\label{sec-orb}

\begin{dfn}
An \emph{orbifold} $X_\varphi$ is a quasi-projective Riemann surface~$X$
with a function $\varphi:X\to\bn$ taking value~$1$ outside a finite
number of points.
\end{dfn}

We may think of a neighborhood of a point $P\in X_\varphi$ with $\varphi(P)$ as the quotient of a 
disk (centered at~$P$) by a rotation of angle $\frac{2\pi}{n}$. A loop around $P$ is considered to be 
trivial in $X_\varphi$ if its lifting bounds a disk. Following this idea, orbifold fundamental groups
can be defined as follows.

\begin{dfn}\label{dfn-group-orb}
For an orbifold $X_\varphi$, let $p_1,\dots,p_n\in X$ be the points such that
$m_j:=\varphi(p_j)>1$. Then, the \emph{orbifold fundamental group} of $X_\varphi$ is defined as
$$
\pi_1^{\text{\rm orb}}(X_\varphi):=\pi_1(X\setminus\{p_1,\dots,p_n\})/\langle\mu_j^{m_j}=1\rangle,
$$
where $\mu_j$ is a meridian of $p_j$. For simplicity, $X_\varphi$ might also be denotes by 
$X_{m_1,\dots,m_n}$ or~$X_{\bar m}$.
\end{dfn}

\begin{example}\label{ex-pres}
If $X$ is a compact surface of genus~$g$ and $\bar m=(m_1,\dots,m_n)$, then
\[
\GG^g_{\bar m}:=\pi_1^{\text{\rm orb}}(X_{\bar m})\!=\!\left\langle\!\!
a_1,\dots,a_g,b_1,\dots,b_g,\mu_1,\dots,\mu_n\left\vert
\prod_{i=1}^g [a_i,b_i]=\prod_{j=1}^n\mu_j, \underset{j=1,\dots,n}{\mu_j^{m_j}=1}\right.\right\rangle\!.
\]
If $X$ is not compact and $\pi_1(X)$ is free of rank~$r$, then
\[
\bff^r_{\bar m}:=\pi_1^{\text{\rm orb}}(X_{\bar m})=\left\langle
a_1,\dots,a_r,\mu_1,\dots,\mu_n\left\vert
\underset{j=1,\dots,n}{\mu_j^{m_j}=1}\right.\right\rangle.
\]
\end{example}

\begin{dfn}
Let $X_\varphi$ be an orbifold and $Y$ a smooth algebraic variety.
A dominant algebraic morphism $f:Y\to X$ defines an \emph{orbifold morphism} $Y\to X_\varphi$
if for all $p\in X$, the divisor $f^*(p)$ is a $\varphi(p)$-multiple. The orbifold~$X_\varphi$
is said to be \emph{maximal} (with respect to $f$) if no divisor $f^*(p)$ is $n$-multiple
for $n>\varphi(p)$.
\end{dfn}

\begin{rem}
\label{rem-maximal}
If $Y$ is a smooth algebraic variety, $X$ is a quasi-projective Riemann surface and $f:Y\to X$
is a dominant algebraic morphism, it is possible to define an orbifold structure $\varphi:X\to\bn$,
where if $p\in X$, $\varphi(p)$ is the $\gcd$ of the multiplicities of the irreducible components
of the divisor~$f^*(p)$. This structure is maximal if and only if $f$ is surjective.
\end{rem}

The following result is well known. Proofs can be found in~\cite{cko,AC-prep}.

\begin{prop}\label{prop-orb}
Let $f:Y\to X$ define an orbifold morphism $Y\to X_\varphi$. Then $f$ induces a morphism
$f_*:\pi_1(Y)\to\pi_1^{\text{\rm orb}}(X_\varphi)$. Moreover, if the generic fiber is connected, then
$f_*$ is surjective.
\end{prop}

Next we compute $\bt_\Pi$ for orbifold groups $\Pi$.

\begin{prop}\label{prop-tg}
If $\Pi=\bff^r_{\bar m}$, then $\bt_\Pi$ is given by the following short exact sequence
\[
1\to \bt_\Pi^\one=(\CC^*)^r \to \bt_\Pi\to\bigoplus_{j=1}^n C_{m_j}\to 1,
\]
(see Remark{~\rm\ref{rem-ses-torus}}) where $C_m$ is the cyclic multiplicative group of $m$-roots of unity.

If $\Pi=\GG^g_{\bar m}$, then $\bt_\Pi$ is given by a similar short exact sequence where the first 
term is $\bt_\Pi^\one=(\bc^*)^{2 g}$ and the last term is the cokernel of the natural mapping 
$C_{m}\to\bigoplus_{j=1}^n C_{m_j}$ where $m:=\lcm\{m_1,\dots,m_n\}$.

\end{prop}

\begin{proof}
It is immediate from the fact that $H=\ZZ^r\oplus C_{m_1}\oplus \dots \oplus C_ {m_n}$ in the
first case and 
$$
H=\ZZ^{2g}\oplus \frac{C_{m_1}\oplus \dots \oplus C_ {m_n}}
{C_m}
$$ 
in the second case.
\end{proof}

Therefore, for orbifolds coming from Riemann surfaces, the components of $\bt_\Pi$ are 
parametrized by the $n$-tuples $\lambda=(\lambda_1,...,\lambda_n)$~of roots of unity 
$\lambda_j$ of order $m_j$ (resp. whose product is 1) if $X$ is non-compact (resp. compact).
Let $\bt_\Pi^\lambda$ denote the component of $\bt_\Pi$ determined by $\lambda$.


\begin{dfn}
The number~$\ell(\lambda)$ of non-trivial coordinates of $\lambda$ is called the \emph{length} of the 
component $\bt_\Pi^\lambda$ of $\bt_\Pi$. If $\xi\in\bt_\Pi^\lambda$, then this number is also called 
the \emph{length} of $\xi$ and it is denoted by~$\ell(\xi)$.
\end{dfn}

Note that there are components of any length~$\ell$, $0\leq\ell\leq n$ for $\bff^r_{\bar m}$, whereas 
this is not the case for $\GG^g_{\bar m}$. The arithmetic of $m_1,\dots,m_n$ imposes some conditions, 
for example, $\bt_{\GG^g_{\bar m}}$ cannot have components of length $1$ and if $m_i$ are pairwise 
coprime, then $\bt_{\GG^g_{\bar m}}$ cannot have components of length $2$.

\begin{dfn}
We define the \emph{$k$-th characteristic variety $\cV_k(X_\varphi)$ of the orbifold $X_\varphi$} as 
the $k$-th characteristic variety $\cV_k(\Pi)$ of its orbifold fundamental group.

Also, if $\xi \in \bt_\Pi$ is a character on $\Pi$, then $H^1(X_\varphi;\CC_\xi)$ will denote 
$H^1(\Pi;\CC_\xi)$.
\end{dfn}

We now compute $\cV_k(X_\varphi)$ for orbifolds $X_\varphi$.
In order to do so, we will follow Example~\ref{ex-cw} by considering the $CW$-complex $K$ associated 
with the presentation of $\Pi:=\pi_1^{\text{\rm orb}}(X_{\varphi})$ given in Example~\ref{ex-pres}.
First we consider the case $\Pi:=\bff^r_{\bar m}$.

\begin{prop}\label{charvar-libre}
Let us consider the group $\Pi:=\pi_1^{\text{\rm orb}}(X_{\varphi})=\bff^r_{\bar m}$. Then,
$$
\cV_k(X_{\varphi}) = 
\begin{cases}
\bt_\Pi&\text{ if } 1\leq k\leq r-1\\
\{\one\}\cup \bigcup \{\bt_\Pi^\lambda \mid {\ell(\lambda)\geq 1}\} & \text{ if } k=r\\
\bigcup\{\bt_\Pi^\lambda \mid {\ell(\lambda)\geq k-r+1}\} & \text{ if } r+1\leq k\leq r+n-1\\
\emptyset & \text{ if } k\geq r+n.
\end{cases}
$$
is a decomposition in irreducible components of $\cV_k(X_{\varphi})$.
\end{prop}

\begin{proof}
First let us consider the case where $\xi\neq \one$. 
Let us consider the complex $C_*(K;\bc)^\xi$; since $\xi\neq\one$, then $\dim\ker\partial_1^\xi=n+r-1$.
The matrix~$M$ for $\partial_2^\xi$ is obtained using Fox calculus and evaluation by~$\xi$.
Let $\lambda:=(\lambda_1,\dots,\lambda_n)$ the $n$-tuple of roots of unity determining the
irreducible component of $\bt_\Pi$ containing~$\xi$.
It is easily seen that
$$
M:=\left(
\array{cccc}
\dfrac{\lambda_1^{m_1}-1}{\lambda_1-1} & 0 & \dots & 0 \\
0 & \dfrac{\lambda_2^{m_1}-1}{\lambda_2-1} & \dots & 0 \\
\vdots & \vdots &   \ddots & \vdots  \\
0 & 0 & \dots & \dfrac{\lambda_n^{m_n}-1}{\lambda_n-1}\\
\\
\hline\\
0 & 0& \dots & 0\\
\dots &\dots& \dots & \dots\\
0 & 0&\dots & 0
\endarray
\right)
\in M((n+r)\times n,\bc);
$$
since $\rk M=n-\ell(\lambda)$, we obtain that
$\dim H^1(K;\bc_{\xi})=r+\ell(\lambda)-1$. 
On the other side, $\dim H^1(K;\bc_\one)=\rk H_1(\Pi)=r$.
\end{proof}

Now we consider the case $\Pi:=\GG^g_{\bar m}$, $2 g+n\geq 2$ ($\GG^g_{\bar m}$ is trivial if $2 g+n< 2$). 
The following corrects a mistake in Delzant's statement~\cite[Proposition 4]{delzant}.

\begin{prop}\label{charvar-compacto}
Let us consider the group $\Pi:=\pi_1^{\text{\rm orb}}(X_{\varphi})=\GG^g_{\bar m}$, $2 g+n\geq 2$. 
Then,
$$
\cV_k(X_{\varphi}) = 
\begin{cases}
\bt_\Pi&\text{ if } 1\leq k\leq 2g-2\\
\{\one\}\cup\bigcup \{\bt_\Pi^\lambda \mid {\ell(\lambda)\geq 2}\} & \text{ if } 2g-1\leq k\leq 2g\\
\bigcup \{\bt_\Pi^\lambda \mid {\ell(\lambda)\geq k-2g+2}\}  & \text{ if } 2 g+1\leq k\leq 2g+n-2\\
\emptyset & \text{ if } k\geq 2g+n-1
\end{cases}
$$
is a decomposition in irreducible components of $\cV_k(X_{\varphi})$.
\end{prop}

\begin{proof}
If $\xi=\one$, $\dim H^1(K;\bc_\one)=\rk H_1(\Pi)=2 g$. Let us assume $\xi\neq\one$ and we assume
the notations of Example~\ref{ex-pres} and Proposition~\ref{charvar-libre}:
$\lambda=(\lambda_1,\dots,\lambda_n)$ where $\lambda_i:=\xi(\mu_i)$; we denote also $x_i:=\xi(a_i)$ and $y_i:=\xi(b_i)$.
We have $\dim\ker\partial_1^\xi=2 g+n-1$ and we obtain the matrix~$M$ for $\partial_2^\xi$ using Fox calculus
and evaluation by~$\xi$:
$$
M:=\left(
\array{ccccc}
\dfrac{\lambda_1^{m_1}-1}{\lambda_1-1} & 0 & \dots & 0 &1  \\
0 & \dfrac{\lambda_2^{m_1}-1}{\lambda_2-1} & \dots & 0 & \lambda_1\\
\vdots & \vdots &   \vdots & 0  & \vdots\\ 
0 & 0 & \dots & \dfrac{\lambda_n^{m_n}-1}{\lambda_n-1}& \lambda_1\cdots \lambda_{n-1}\\
\\
\hline\\
0 & 0& \dots & 0& y_1-1\\
0 & 0& \dots & 0& 1-x_1\\
\dots &\dots& \dots & \dots & \dots\\
0 & 0&\dots & 0& y_g-1\\
0 & 0&\dots & 0& 1-x_g
\endarray
\right)
\in M((n+2 g)\times (n+1),\bc);
$$
since $\rk M=n+1-\ell(\lambda)$, we obtain that
$\dim H^1(K;\bc_{\xi})=2 g+\ell(\lambda)-2$. The result follows since
the case $\ell(\lambda)=1$ cannot arise.
\end{proof}

As an immediate corollary of Propositions~\ref{charvar-libre} and \ref{charvar-compacto}, 
the twisted cohomologies of a character, its inverse, and its conjugate can be related as follows.

\begin{prop}\label{prop-inverse-riemann}
Let $X_\varphi$ be an orbifold, let $\Pi:=\pi_1^{\text{\rm orb}}(X_\varphi)$,
and let $\xi\in\bt_\Pi$, then 
$$\dim H^1(\Pi;\CC_\xi)=\dim H^1(\Pi;\CC_{\xi^{-1}})=\dim H^1(\Pi;\CC_{\bar \xi}).$$
\end{prop}

\begin{proof}
In the proofs of Propositions~\ref{charvar-libre} and \ref{charvar-compacto} it was 
computed that 
$$
\dim H^1(\Pi;\CC_\xi) = 
\begin{cases}
\rk H_1(\Pi)&\text{ if } \xi=\one\\
\ell(\lambda)-\chi(X) & \text{ if } \xi\in\bt_\Pi^\lambda.
\end{cases}
$$
Clearly, if $\xi\in\bt_\Pi^\lambda$ then $\xi^{-1}\in\bt_\Pi^{\lambda^{-1}}$ 
and $\bar\xi\in\bt_\Pi^{\bar\lambda}$. Moreover $\ell(\lambda)=\ell(\lambda^{-1})=\ell(\bar\lambda)$,
and the statement follows.
\end{proof}

As a consequence of Propositions~\ref{charvar-libre} and~\ref{charvar-compacto}
the twisted cohomology of an orbifold 
$X_\varphi$ can be identified with the twisted cohomology of a Riemann surface. 

\begin{prop}\label{prop-orbi-riemann}
Let $X_\varphi$ be an orbifold with singular points $p_1,\dots,p_n$ and orbifold 
fundamental group $\Pi:=\pi_1^{\text{\rm orb}}(X_\varphi)$.
For $\xi\in\bt_\Pi$, 
set $Y:=X\setminus\{p_j\in X\mid \xi(\mu_j)\neq 1\}$. Let $\xi_1$ be the
character on $Y$ determined by~$\xi$. Then
$H^1(X_\varphi;\bc_\xi)$ is naturally identified with~$H^1(Y;\bc_{\xi_1})$.
\end{prop}

\begin{proof}
The particular cases of Propositions~\ref{charvar-libre} and \ref{charvar-compacto} 
where no orbifold points are present give the dimensions of the twisted cohomology 
groups of a Riemann surface $Y$: 
$$
\dim H^1(Y;\CC_\xi) = 
\begin{cases}
\rk H_1(Y)&\text{ if } \xi=\one\\
-\chi(Y)  & \text{ if } \xi\neq\one.
\end{cases}
$$

Now fix $J$ a subset of $\{1,\dots,n\}$ of size $l$ and suppose $\xi$ is such that 
$\xi(\mu_j)\neq 1$ precisely when $j\in J$. Then $\xi\in\bt_\Pi^\lambda$ and $\ell(\lambda)=\ell$.
We then have that 
$$
\dim H^1(X_\varphi;\bc_\xi) = 
\begin{cases}
\rk H_1(\Pi)&\text{ if } \xi=\one\\
\ell-\chi(X) & \text{ if } \xi\neq\one.
\end{cases} 
$$
Clearly $Y=X\setminus\{p_j \mid j\in J\}$ is a Riemann surface of Euler characteristic
$\chi(Y)=\chi(X)-\ell$ and first Betti number $\rk H_1(Y)=\rk H_1(X)+\ell=\rk H_1(\Pi)+\ell$.
It follows that $\dim H^1(Y;\CC_\xi)=\dim H^1(X_\varphi;\bc_\xi)$,
and so the restriction $H^1(Y;\CC_\xi)\to H^1(X_\varphi;\bc_\xi)$ is an isomorphism.
\end{proof}

\section{Deligne's theory and Hodge-like decompositions}\label{sec-deligne}

In what follows, we briefly summarize Deligne's results~\cite{del:70} (with some addenda by Timmerscheidt~\cite{timm:87})
on twisted cohomology of quasi-projective varieties. Let $\xi\in\bt_G$, $G:=\pi_1(X)$. Consider
the line bundle $L_\xi:=\bc_\xi\otimes\cO_X$ over~$X$. The local system of coefficients $\bc_\xi$ induces a 
flat connection $\nabla$ on $L_\xi$. Let us fix $\bar{L}_\xi$ an extension of $L_\xi$ to $\bar{X}$, whose 
associated flat connection~$\bar{\nabla}$ is meromorphic (having log poles along $\cD$) and extends~$\nabla$. Note that plenty of such extensions 
are possible (there is a choice of a logarithm determination around every component $D$ of $\cD$). More precisely,
fix an irreducible component $D$ of $\cD$ and let $p\in D\setminus\sing(\cD)$. Let $u,v$ be a local analytical 
system of coordinates centered at~$p$ such that~$v=0$ is the local equation of $D$. Let $\mu_D$ be a meridian of 
$D$ and let $\xi(\mu_D)=:t$. The extension $\bar{\nabla}$ to $D$ is determined by the choice of $\alpha\in\bc$
such that $\exp(2 \sqrt{-1} \pi\alpha)=t$, or equivalently, such that $v^\alpha$ is the equation of 
a~\emph{multivalued} flat section on a suitable chart of $\bar{L}_\xi$.

\begin{dfn}
We say that $\alpha$ is the \emph{residue} of the meromorphic extension $\bar{\nabla}$ around~$D$.
\end{dfn}

\begin{dfn}
An extension $\bar{L}_\xi$ as above is said to be~\emph{suitable} if 
the residues of $\bar{L}_\xi$ around the components of $\cD$ are not positive integers.
The \emph{Deligne's extension} of $(L_\xi,\nabla)$ is the unique holomorphic extension of $L_\xi$ whose
associated meromorphic flat connection (with log poles) 
is so that its residues around any component of $\cD$ have real parts in $[0,1)$. 
Such an extension will be denoted by~$\tilde{L}_\xi$.
\end{dfn}

The main result proved by Deligne in~\cite{del:70} states that, 
if $\bar{L}_\xi$ is suitable then the hypercohomology of the twisted complex of 
holomorphic sheaves of logarithmic forms with poles along $\cD$ 
(denoted by $\Omega_{\bar{X}}^\bullet(\log\cD)\otimes\bar{L}_\xi$) is isomorphic to the twisted cohomology 
of $X$ with coefficients in $\xi$, that is,
\begin{equation}
\label{eq-hypercohomology}
\mathbb H^i(\bar X;\Omega_{\bar{X}}^\bullet(\log\cD)\otimes\bar{L}_\xi)\cong H^i(X;\bc_\xi).
\end{equation}
This induces a decomposition $H^1(X;\bc_\xi)=H^{\cO}_\xi\oplus H^{\overline{\cO}}_\xi$, where
$H^{\cO}_\xi$ corresponds to the $(1,0)$-term and $H^{\overline{\cO}}_\xi$ corresponds to the $(0,1)$-term.
In a nutshell, Timmerscheidt~\cite{timm:87} showed that, in the case of Deligne's extensions of unitary bundles, 
the associated spectral sequence degenerates in the first step. Next, we will describe some of the main properties
derived from this result, some of which are particular to the Deligne's extension.

\begin{thm}[\cite{del:70,timm:87}]\label{prop-dlg}
The following properties hold:
\begin{enumerate}
\enet{\rm(\arabic{enumi})}
\item\label{prop-dlg-pz} If $\bar{L}_\xi$ is a suitable extension, then 
the space $H^{\cO}_\xi$ is the homology of the complex
\begin{equation}\label{spseq}
H^0(\bar{X};\bar{L}_\xi)\overset{\bar{\nabla}}{\longrightarrow} H^0(\bar{X};\Omega_{\bar{X}}^1(\log\cD)\otimes\bar{L}_\xi)
\overset{\bar{\nabla}}{\longrightarrow}H^0(\bar{X};\Omega_{\bar{X}}^2(\log\cD)\otimes\bar{L}_\xi)
\end{equation}
and $H^{\overline{\cO}}_\xi$ is the kernel of 
\begin{equation}\label{spseq1}
\bar{\nabla}:H^1(\bar{X};\bar{L}_\xi)\to H^1(\bar{X};\Omega_{\bar{X}}^1(\log\cD)\otimes\bar{L}_\xi). 
\end{equation}
\item\label{prop-dlg-unit} If $\xi$ is unitary and $\tilde{L}_\xi$ is the Deligne extension then 
$\bar{\nabla}=0$ in \eqref{spseq} and \eqref{spseq1}, i.e.,
$$
H^{\cO}_\xi=H^0(\bar{X};\Omega_{\bar{X}}^1(\log\cD)\otimes\tilde{L}_\xi), \quad
H^{\overline{\cO}}_\xi=H^1(\bar{X};\tilde{L}_\xi).
$$
\end{enumerate}
\end{thm}

\begin{rem}\label{rem-mhs}
This decomposition is, in general, non-canonical. As in Theorem~\ref{prop-dlg}\ref{prop-dlg-unit}
more properties can be derived when $\xi$ is unitary. Following the ideas in~\cite[III-IV]{ara:97},
the decomposition  $H^1(X;\bc_\xi)=H^{\cO}_\xi\oplus H^{\overline{\cO}}_\xi$ is natural and carries
a mixed Hodge structure.
\end{rem}

\begin{dfn}
\label{def-pure}
An element
$0\neq\theta\in H^1(X;\bc_\xi)$ is said to be \emph{holomorphically} (resp. \emph{anti-holomorphically})
\emph{pure} if $\theta\in H^{\cO}_\xi$ (resp. $\theta\in H^{\overline{\cO}}_\xi$).
\end{dfn}

The following is yet another consequence of the Hodge theory on the cohomology of $X$, which will be very 
useful for our purposes. The statement appears in the proof of~\cite[Proposition~V.1.4]{ara:97} (where
$X$ must be replaced by~$\bar{X}$ in the last summand).

\begin{prop}\label{hd}
There is a natural real decomposition
\begin{equation*}
H^1(X;\bc)=(H^{1 0}(X)\oplus H^{1 1}_{\br}(X))\oplus
\sqrt{-1}(H^{1 1}_{\br}(X)\oplus H^1(\bar{X};\br)).
\end{equation*}
The sum of the first three terms corresponds to $H^0(\bar{X};\Omega_{\bar{X}}^1(\log\cD))$, whereas 
the sum of the first two corresponds to those forms having purely imaginary residues. The residues 
along the components of $\cD$ of the forms in the first and last terms are trivial.
\end{prop}

\begin{para}\textbf{Character Decompositions.} \cite{ara:97,be}\label{char-dec}
Let $\xi\in\bt_G=H^1(X;\bc^*)$. Note that there exists a torsion element $\tau$ such that 
$\tilde{\xi}:=\tau^{-1}\xi\in\bt_G^\one$ and there exists $\eta\in H^1(X;\bc)$ such that 
$\tilde{\xi}=\exp(\eta)$. The element $\eta$ is unique up to sum by an element in $2 \sqrt{-1}\pi H^1(X;\bz)$.
According to Proposition~\ref{hd}, there exist $\omega\in H^0(\bar{X};\Omega_{\bar{X}}^1(\log\cD))$ and 
$\delta\in H^1(\bar{X};\br)$ such that $\eta=\omega+\sqrt{-1}\delta$.
Summarizing, $\xi=\psi\exp(\omega)$, where $\psi:=\tau\exp(\sqrt{-1}\delta)$ is unitary.

Note that any choice of $\omega_0\in H^{1 1}_{\br}(X)$ leads to another decomposition $\xi=\tilde{\psi}\exp(\tilde{\omega})$
where $\tilde{\omega}:=\omega-\sqrt{-1}\omega_0$ and $\tilde{\psi}:=\tau\exp(\sqrt{-1}(\delta+\omega_0))$.
\end{para}

\begin{dfn}
A decomposition $\xi=\psi\exp(\omega)$ is called a \emph{unitary-holomorphic} decomposition of $\xi$
if $\psi$ is unitary and $\omega\in H^0(\bar{X};\Omega_{\bar{X}}^1(\log\cD))$.
Such a decomposition is called:
\begin{itemize}
\item\emph{integrally unramified} if $\cD^\psi=\cD^\xi$ and $\omega$ is holomorphic outside $\cD^\xi$;
\item \emph{strict} if it is integrally unramified and $\omega\notin 2 \sqrt{-1}\pi H^1(X;\bz)$.
\end{itemize}
\end{dfn}

\begin{rem}
In the definition of \emph{integrally unramified}, the condition of being holomorphic outside $\cD^\xi$
is non-void. Let $\xi=\psi\exp(\omega)$  be an integrally unramified
decomposition and let $\eta$ be  a logarithmic one-form having integral residues around the components
of $\cD-\cD^\xi$. The decomposition $\xi=\psi\exp(\omega+\eta)$ is also unitary-holomorphic
but not integrally unramified.
\end{rem}

\begin{rem}\label{rem-nabla}
Let us fix a unitary-holomorphic decomposition $\xi=\psi\exp(\omega)$. We consider the Deligne extension $\tilde{L}_\psi$
associated with $\psi$. This is also an extension for $L_\xi$ and the meromorphic connection is
$\nabla_\omega:=\bar{\nabla}+\wedge\omega$ (see, e.g.,~\cite[Section~V]{ara:97} or~\cite{be}), where
$\bar{\nabla}$ is the connection associated with~$\psi$: it is a flat meromorphic connection
extending $\nabla$ and its monodromy equals~$\xi$ (using the unitary-holomorphic decomposition). 
\end{rem}

\begin{lem}\label{lem-unram}
Let $\xi=\psi\exp(\omega)$ be an integrally unramified unitary-holomorphic decomposition.
\begin{enumerate}
\item\label{lem-unram2} Any integral residue of~$\nabla_\omega$ vanishes
(and, in particular, Theorem{\rm~\ref{prop-dlg}\ref{prop-dlg-pz}} can be applied).
\item\label{lem-unram3} The character $\psi$ (resp. the form $\omega$) is a restriction of a unitary character
(resp. a logarithmic form) defined on $X^\xi$.
\end{enumerate}
\end{lem}

\begin{proof}
Let $D$ be an irreducible component of $\cD$ where $\xi$ does
not ramify. From the definition of \emph{integrally unramified}
we deduce that $\exp(\omega)$ does not ramify along~$D$, and hence,
the same happens for~$\psi$. Therefore~\eqref{lem-unram2} and~\eqref{lem-unram3} follow.
\end{proof}

%

The proofs of the results of \S\ref{sec-char} are easier for characters
admitting a strict unitary-holomorphic decomposition, see Corollary~\ref{cor-caso1}. The following result
shows under which conditions such decompositions exist. The sufficient condition~\ref{lema-char1} 
is classical and it is well known in the projective case while the second sufficient condition
is original in this context to our knowledge. 

\begin{lem}\label{lema-char}
In the following cases the character~$\xi$ admits a strict unitary-holomorphic decomposition:
\begin{enumerate}
\enet{\rm(\arabic{enumi})} 
\item\label{lema-char1} $\xi$ is non-unitary
\item\label{lema-char2} $b_1(X^\xi)>b_1(\bar{X})$.
\end{enumerate}
\end{lem}

\begin{proof}
By Lemma~\ref{lem-unram}\eqref{lem-unram3} it is equivalent to find a strict unitary-holomorphic 
decomposition for the induced character~$\xi_0$ in $X^\xi$. Then, replacing $X$ by $X^\xi$,
we may assume that $\xi$ ramifies along every irreducible component of~$\cD$. 
Consider $\xi=\psi\exp(\omega)$ a unitary-holomorphic decomposition of $\xi$.

Following Arapura~\cite{ara:97} one can choose $\omega_1\in H^1(X;\ZZ)$ with non-trivial residues
along~$\cD$. Consider $\eta=\alpha+\beta$ a decomposition where
$\alpha\in H^1(\bar X;\RR)$ and $\beta\in H^{11}_\RR(X)$. Note that 
$\xi=\psi\exp(\omega)\exp(2\pi\sqrt{-1} \eta)=
\psi\exp(2\pi\sqrt{-1} (\alpha+t\beta)) \exp(\omega+2\pi\sqrt{-1} (1-t)\beta)$. Note that 
$\psi_1:=\psi\exp(2\pi\sqrt{-1} (\alpha+t\beta))$ has non-trivial residues along $\cD$.
Eventually replacing $\psi$ by $\psi_1$ and $\omega$ by $\omega+2\pi\sqrt{-1} (1-t)\beta$ one might 
assume that $\xi=\psi\exp(\omega)$ is an integrally unramified unitary-holomorphic decomposition of $\xi$.

All is left to check is that this can be done choosing $\omega\not\equiv 0\mod H^1(X;\bz)$.

If~\ref{lema-char1} holds, then $\omega \notin H^1(X;\bz)$ (by Proposition~\ref{hd}), 
otherwise $\xi=\psi$ would be a unitary character. 

If~\ref{lema-char2} holds, then Proposition~\ref{hd}
implies that $H^{1 1}_{\br}(X)\neq 0$. 
A generic choice of $\omega_0\in H^{1 1}_{\br}(X)$ leads to another decomposition where 
$\psi$ (resp. $\omega$) is replaced by $\psi\exp(\sqrt{-1}\omega_0)$ (resp. $\omega-\sqrt{-1}\omega_0$)
satisfying $\omega\notin H^1(X;\bz)$.
\end{proof}

\section{Anti-holomorphic pure factors of twisted cohomology}\label{sec-inverse}

In this section we study the relationship between the twisted cohomologies relative
to characters $\xi$ and~$\xi^{-1}$ taking into account Deligne's theory. In particular,
we study the properties of $H_{\xi^{\pm 1}}^\cO$ and $H_{\xi^{\pm 1}}^{\overline{\cO}}$ 
(see~\eqref{eq-hypercohomology} and paragraph right after for a definition).

\begin{thm}[\cite{timm:87,ara:97}]\label{prop-tm}
If the character~$\xi$ ramifies along each irreducible component of~$\cD$, then there is a natural
inclusion $H^{\overline{\cO}}_\xi\hookrightarrow \overline{H^{\cO}_{\xi^{-1}}}$.
\end{thm}

Theorem~\ref{prop-tm} is proved in \cite[Theorem~5.1]{timm:87} for unitary characters
(where the inclusion is in fact an isomorphism). The second part of the proof 
of~\cite[Proposition~V.1.4]{ara:97} provides the general result. We can summarize it as follows.
For $\xi=\psi\exp(\omega)$, let $\alpha\in H^{\overline{\cO}}_\xi$, that is, 
$\alpha\in H^1(\bar X,\tilde{L}_\psi)$ and $\alpha\wedge \omega=0$, 
cf. Remark~\ref{rem-nabla} and exact sequence in~\eqref{spseq1}. 
Since $\xi$ ramifies along each irreducible component of~$\cD$, one has 
$\tilde{L}_{\psi}^{-1} = \tilde{L}_{\psi^{-1}}\otimes \mathcal O_{\bar X}(-\cD)$ 
(cf.~\cite[Theorem~5.2]{timm:87}). Therefore, 
$\bar \alpha\in H^0(\bar X,\Omega_{\bar X}^1(\log \cD)\otimes \tilde{L}_{\psi^{-1}}).$
According to the $L^2$~cohomology arguments in the proof of~\cite[Proposition~V.1.4]{ara:97} 
one has that $\bar \alpha \wedge \omega=0$ and if $\tilde{L}_\psi$ is trivial, therefore $\bar \alpha$ 
is not a multiple of $\omega$. 
Thus $0\neq\bar \alpha\in H^{\cO}_{\xi^{-1}}$ and the statement follows, see~\eqref{spseq}.

Note that this theorem does not deal with arbitrary characters, but only with those that ramify 
along every irreducible component of $\cD$. In the general case one might have to resort to the
conjugated character as we will see in what follows. Let $\xi$ be a character and consider
as in~\S\ref{sec-prelim} the divisor~$\cD^\xi$ (containing the components
of $\cD$ where $\xi$ ramifies) and~$X^\xi:=\bar{X}\setminus\cD^{\xi}$.
For the sake of clarity we denote by $\xi_0$ the character induced by $\xi$ on~$X^\xi$.   

\begin{prop}\label{prop-iso}
For any character $\xi$ on $X$ there is a natural isomorphism 
$H^{\overline{\cO}}_{\hat{\xi}}\cong H^{\overline{\cO}}_{\hat{\xi}_0}$
where $\hat{\xi}_0$ is the character induced by~$\hat{\xi}$ on~$X^{\hat{\xi}}$ and 
$\hat{\xi}$ is either $\xi$ or its conjugate~$\bar \xi$.
\end{prop}

In order to prove this proposition we need to recover more information on the
Deligne's decomposition of the twisted cohomology of a Riemann surface, which has already been
computed in~\S\ref{sec-orb}. The proof will be postponed to the end of the section.
Note also that Proposition~\ref{prop-orbi-riemann} allows to extend the concept of decomposition into a holomorphic and 
an anti-holomorphic part to $H^1(X_\varphi;\CC_\xi)$ for orbifolds.

We consider now some computations of meromorphic extensions for the particular case
$X:=\bp^1\setminus\{p_1,\dots,p_n\}$, $n>0$, i.e., $\cD\neq \emptyset$ is the reduced divisor supported on 
$\{p_1,\dots,p_n\}$. The group $G:=\pi_1(X)$ is generated by meridians
$\mu_j$, $j=1,\dots,n$. For a suitable choice of these meridians the only relation
is $\mu_1\cdot\ldots\cdot \mu_n=1$. Let us fix a character $\xi\in\bt_G$; an extension $\bar{L}_\xi$ to $\bp^1$
of $L_\xi:=\bc_\xi\otimes\cO_X$ (with a meromorphic extension $\bar{\nabla}$ of the connection of $L_\xi$)
is determined by the choice of $\alpha_j\in \CC$, $j=1,\dots,n$, 
such that $\xi(\mu_j)=\exp(-2\sqrt{-1}\pi\alpha_j)$. Note that $k:=-\sum_{j=1}^n \alpha_j\in\bz$. From these choices,
the line bundle $\bar{L}_\xi$ admits a \emph{multivalued flat meromorphic section}~$\sigma$ having \emph{complex order} 
$\alpha_j$ at $p_j$. We deduce from this fact that $\bar{L}_\xi\cong\cO_{\bp^1}(-k)$.
Using this idea for arbitrary Riemann surfaces we obtain this useful result which is well known.

\begin{prop}\label{prop-neg}
Let $X$ be a quasi-projective Riemann surface with compactification~$\bar{X}$, $n:=\#(\bar{X}\setminus X)$. 
Let $\xi$ be a character on $X$ and let $\tilde{L}_\xi$ be the Deligne extension of~$L_\xi$ to~$\bar X$. 
The following results hold:
\begin{enumerate}
\enet{\rm(\arabic{enumi})} 
\item If $n=0$, then $\deg\tilde{L}_\xi=0$;
\item if $n>0$, then $-n< \deg\tilde{L}_\xi\leq 0$;
\item if $n>0$ and $\xi$ is unitary, then $\xi$ is trivial if and only if $\deg\tilde{L}_\xi= 0$;
\item if $n=0$ and $\xi$ is unitary, then $\xi$ is trivial if and only if $\tilde{L}_\xi\cong\cO_{\bar{X}}$;
\item $\tilde{L}_\xi$ admits non-zero holomorphic sections if and only if $\tilde{L}_\xi\cong\cO_{\bar{X}}$.
\end{enumerate}
\end{prop}

Let us continue for a moment our discussion above where the Deligne extension $\tilde{L}_\xi$ for the unitary 
character $\xi$ has been fixed (hence $\alpha_j\in\RR$). If $\xi$ is not the trivial character, then $0<k<n$.
Using Theorem~\ref{prop-dlg}\ref{prop-dlg-unit}, one has 
$H^{\cO}_\xi=H^0(\bp^1;\Omega_{\bp^1}^1(\log\cD)\otimes\tilde{L}_\xi)$. 
Since $\Omega_{\bp^1}^1(\log\cD)\otimes\tilde{L}_\xi=\cO_{\bp^1}(-2+n-k)$,  one can deduce that $\dim H^{\cO}_\xi=n-k-1$.
Also $H^{\overline{\cO}}_\xi=H^1(\bp^1;\tilde{L}_\xi)$; by Serre
duality, this space has the same dimension as $H^0(\bp^1;\Omega_{\bp^1}^1\otimes(\tilde{L}_\xi)^{-1})$
and $\Omega_{\bp^1}^1\otimes(\tilde{L}_\xi)^{-1}=\cO_{\bp^1}(-2+k)$. Hence $\dim H^{\overline{\cO}}_\xi=k-1$.
Also note that $\dim H^{\overline{\cO}}_{\bar{\xi}}=(n-k)-1$ by Serre duality, which agrees with the
isomorphism in Theorem~\ref{prop-tm} for the unitary case. See also~\cite{li:09} for this kind of computations.


Following these ideas, a more  general result holds.

\begin{prop}\label{prop-hol-rs}
Let $X$ be a quasi-projective Riemann surface, such that $\bar{X}$ is a curve of genus~$g$, 
and $\cD:=\bar{X}\setminus X$ is a reduced effective divisor of cardinality $n$.  
Let $G:=\pi_1(X)$ and fix $\xi\in\bt_G\setminus\{\one\}$. Let $k$ be the integer which is the
negative of the sum of residues associated with~$\xi$.
Then, $H^{\cO}_\xi\neq 0$ except in the following cases:
\begin{enumerate}
\enet{\rm(\arabic{enumi})}
\item $g=0$ and $n\leq 2$;
\item $g=1$ and $n=0$;
\item $g=0$, $n\geq 3$ and $k=n-1$ (in particular, $X=X^\xi$);
\item $g=1$, $n=1$ and $\xi=\exp(\omega)$ for $\omega\in H^0(\bar{X},\Omega^1_{\bar{X}}(\log\cD))=
H^0(\bar{X},\Omega^1_{\bar{X}})=H^0(\bar{X},\cO_{\bar{X}})$.
\end{enumerate}
\end{prop}

\begin{proof}
The first two cases are immediate since $G$ is abelian.

We recall, see \eqref{spseq}, that $H^{\overline{\cO}}_\xi$ is the cokernel of 
$$
A:=H^0(\bar{X},\tilde{L}_\xi)\overset{\bar{\nabla}}{\longrightarrow} 
H^0(\bar{X},\Omega_{\bar{X}}^1(\log\cD)\otimes\tilde{L}_\xi)=:B,
$$
where $\tilde{L}_\xi$ is, as usual, the Deligne extension of $L_\xi$ to $\bar X$.
Recall that $a:=\deg\tilde{L}_\xi=-k$ and $b:=\deg\Omega_{\bar{X}}^1(\log\cD)\otimes\tilde{L}_\xi=2 g-2+n-k$;
moreover, $0\leq k< n$ if $n>0$ and $k=0$ if $n=0$. By Proposition~\ref{prop-neg}, $A\neq 0$ if and only if
$\tilde{L}_\xi\cong\cO_{\bar{X}}$, in particular, $\dim A=1$. We distinguish several cases.

If $g>1$, since $b>0$ then $B\neq 0$. Then, we get the statement when $A=0$. If $A\neq 0$, it is enough
to prove that $\dim B>1$. In this case, $B$ is the space of holomorphic sections of $\Omega_{\bar{X}}^1(\log\cD)$
which admits the space of holomorphic $1$-forms of $\bar{X}$ as a subspace. Since such a space has dimension $g>1$ 
we are done.

For $g=1$, we assume $n>0$. Hence $0\leq k<n$ and  $b> 0$; if $A=0$ we are done. 
If $A\neq 0$, we have in particular that $B$ is the space of holomorphic sections of 
$\Omega_{\bar{X}}^1(\log\cD)=\cO_{\bar{X}}(\cD)$.
If $n\geq 2$, then $\dim B>1$ and we are done. If $n=1$, then $\dim B=1$ and hence $H^{\cO}_\xi=0$. 
Note that $\tilde{L}_\xi\cong\cO_{\bar{X}}$ is equivalent to the property $\xi=\exp(\omega)$ for 
$\omega\in H^0(\bar{X},\Omega^1_{\bar{X}}(\log\cD))$.

For $g=0$, we assume $n>2$ in which case $b=n-k-2\geq -1$ and hence $\dim B=n-k-1$.
If $k=0$, we have immediately that $\dim B>1=\dim A$. If $k>0$, then $A=0$ and
if  $k<n-1$, then $\dim B>0$.
\end{proof}

\begin{cor}\label{cor-hol-1}
Let us assume that $H^1(X;\bc_\xi)\neq 0$ and $H^{\cO}_\xi= 0$. Consider
$p\in X$, $Y:=X\setminus\{p\}$, and $\xi_1$ the induced character in~$Y$.
Then, $H^{\cO}_{\xi_1}\neq 0$ (on $Y$).
\end{cor}

\begin{proof}
If $g=1$, it is trivial. If $g=0$, note that the maximum value of $k$ equals $n-1$ 
and it can be obtained only when $\xi$ ramifies around all the punctures.
\end{proof}

\begin{example}\label{ex-elliptic}
In Example~\ref{ex-notram}, we showed that the twisted cohomology of a Riemann surface $X\subset\bar{X}$ with 
respect to a character $\xi$ which does not ramify around a point $p\in\bar{X}$
changes if we replace $X$ by $X\cup\{p\}$. This example aims the same purpose, but is more subtle and points
at the root of the relation between the twisted cohomology of a character $\xi$ on $X$ and $\xi_0$ on $X^\xi$.

Let $\bar X$ be an elliptic curve, $p$ a point on it and $X=\bar X\setminus\{p\}$.
The space of characters over $X$ and $\bar{X}$ coincide. Let us fix $\one\neq\xi\in\bt_{\pi_1(X)}$ and let $\xi_0$ be
the corresponding character in $\bt_{\pi_1(\bar{X})}$. Since $\pi_1(\bar{X})$ is Abelian, $H^1(\bar{X};\bc_{\xi_0})=0$.

Let us decompose $\xi=\psi\exp(\omega)$, where $\psi$ is unitary and
$\omega\in H^0(\bar{X},\Omega_{\bar{X}}^1(\log p))= H^0(\bar{X},\Omega_{\bar{X}}^1)$. The Deligne extension
of $\psi$ is a degree~$0$ line bundle $\tilde{L}_\psi$ over $\bar{X}$ which is trivial if and only if $\psi=\one$.
The short exact sequence 
$$
0\to \Omega_{\bar{X}}^1\otimes\tilde{L}_\psi\to\Omega_{\bar{X}}^1(\log p)\otimes\tilde{L}_\psi\to\bc_p\to 0
$$
induces (note that $\Omega_{\bar{X}}^1\cong\cO_{\bar{X}}$):
\begin{equation*}
0\to H^0(\bar{X};\tilde{L}_\psi)\to 
H^0(\bar{X};\cO( p)\otimes\tilde{L}_\psi)\to\bc
\to H^1(\bar{X};\tilde{L}_\psi)\to H^1(\bar{X};\cO( p)\otimes\tilde{L}_\psi)\to 0.
\end{equation*}
Let us assume that $\psi\neq\one$. Then, this sequence implies (using Serre Duality and the Riemann-Roch formula):
\begin{equation*}
H^0(\bar{X};\tilde{L}_\psi)=H^1(\bar{X};\tilde{L}_\psi)=H^1(\bar{X};\cO( p)\otimes\tilde{L}_\psi)=0,\quad
H^0(\bar{X};\cO( p)\otimes\tilde{L}_\psi)\cong \bc.
\end{equation*}
Note also that $H^0(\bar{X};\tilde{L}_\psi)=0=H^1(\bar{X};\tilde{L}_\psi)$. Applying
\eqref{spseq} and \eqref{spseq1} we obtain $H_\xi^{\cO}\cong\bc$
and $H_\xi^{\overline{\cO}}=0$.

If $\psi=\one$ the sequence implies:
\begin{equation*}
H^0(\bar{X};\cO)\cong H^0(\bar{X};\cO( p))\cong H^1(\bar{X};\cO)\cong\bc,\quad
 H^1(\bar{X};\cO( p))=0.
\end{equation*}
Since $H^0(\bar{X};\cO)\cong H^1(\bar{X};\cO)\cong\bc$, applying
\eqref{spseq} and \eqref{spseq1} we obtain $H_\xi^{\cO}=0$
and $H_\xi^{\overline{\cO}}\cong\bc$.

In both cases we obtain that $H^1(X;\bc_\xi)\cong\bc$ but the decomposition into the
spaces $H_\xi^{\cO}$ and $H_\xi^{\overline{\cO}}$ depends on the analytic type of $\bar{X}$.
\end{example}

The situation described in the previous example holds in a more general setting. We
are in position to state the last ingredient needed for the proof of Proposition~\ref{prop-iso}.

\begin{prop}\label{prop-iso0}
Let $X$ be a Riemann surface and let $\xi$ be a character on $X$
such that the Deligne extension $\tilde{L}_\psi$ associated with the unitary part of 
a unitary-holomorphic decomposition of $\xi$ is non-trivial. 
Then
$H^{\overline{\cO}}_{\xi}= H^{\overline{\cO}}_{\xi_0}=H^1(\bar{X};\tilde{L}_\psi)$
where $\xi_0$ is the character induced by~$\xi$ on~$X^{\xi}$.
\end{prop}

\begin{proof}
Let us consider the chain $X\subset X^\xi\subset\bar{X}$ and the reduced divisors
$\cD:=\bar{X}\setminus X$ and $\cD^\xi:=\bar{X}\setminus X^\xi$. Since the statement
is trivial when $\cD=\cD^\xi$ we may assume $\cD^\xi\subsetneq\cD$.

We must consider the sequence \eqref{spseq1} applied to $\cD$ and $\cD^\xi$ for the Deligne
extension $\tilde{L}_\psi$ associated with $\xi=\psi\exp(\omega)$, where $\psi$ is unitary
and $\omega\in H^0(\bar{X};\Omega_{\bar{X}}^1(\log\cD^\xi))$. By hypothesis $\tilde{L}_\psi\not\cong\cO_{\bar{X}}$.

This sequence reduces to a morphism, defined by the exterior product by $\omega$,
where the source $H^1(\bar{X};\tilde{L}_\psi)$ is common for $\cD$ and $\cD^\xi$. 
Let us study the targets.
For $\cD$ we have $H^1(\bar{X};\Omega_{\bar{X}}^1(\log\cD)\otimes\tilde{L}_\psi)$, 
which is isomorphic, by Serre Duality, to the dual of 
$H^0(\bar{X};\cO(-\cD)\otimes\tilde{L}_\psi^{-1})$. Since $\deg\cD>0$
and $\deg\tilde{L}_\psi^{-1}< \deg\cD$, see Proposition~\ref{prop-neg}, this space is trivial.

Let us now consider $H^1(\bar{X};\Omega_{\bar{X}}^1(\log\cD^\xi)\otimes\tilde{L}_\psi)$, i.e., 
by duality $H^0(\bar{X};\cO(-\cD^\xi)\otimes\tilde{L}_\psi^{-1})$. Note that the degree $e$
of this line bundle is non-positive. If $e<0$, then the space is trivial again. If $e=0$, then 
we have $\cD^\xi=\emptyset$; since $\tilde{L}_\psi$ is non-trivial, the space is also trivial.

Then, $H^{\overline{\cO}}_{\xi}= H^1(\bar{X};\tilde{L}_\psi)= H^{\overline{\cO}}_{\xi_0}$.
\end{proof}


\begin{proof}[Proof of Proposition{\rm~\ref{prop-iso}}]
The result is trivial for unitary characters, so we assume that $\xi$
is non-unitary. We break the proof in several steps.

\begin{step}\label{step1}
If $X$ is a Riemann surface. 
\end{step}

After Proposition~\ref{prop-iso0} it is enough to prove that either $\xi$ or $\bar{\xi}$
admits a unitary-holomorphic decomposition such that 
the Deligne extension $\tilde{L}_\psi$ associated with the unitary part is non-trivial.

Let us assume that it is not the case for $\xi$. 
Then, $\cD^\xi=\emptyset$, $\bar{X}=X^\xi$ (see Remark~\ref{rem-nabla}), the unitary part 
of $\xi$ is $\one$ (see Proposition~\ref{prop-neg}) and $\xi=\exp(\omega)$, 
where $0\neq\omega\in H^0(\bar{X};\Omega_{\bar{X}}^1)$. We deduce that
$\bar{\xi}=\exp(\bar{\omega})$, where $\bar{\omega}\in H^1 (\bar{X};\cO_{\bar X})$.

Let us consider the decomposition $\bar{\omega}=\alpha+\beta$ associated with 
$$
H^1(\bar{X};\bc)=2\sqrt{-1}\pi H^1(\bar{X};\br)\oplus H^0(\bar{X};\Omega_{\bar{X}}^1).
$$
Note that the spaces 
$H^{\overline{\cO}}_{\xi}$ and $H^{\overline{\cO}}_{\xi_0}$ do not change if 
we replace $\xi$ by $\exp(t\omega)$, $t\in\br^*$; we may assume that 
$\alpha\notin 2\sqrt{-1}\pi H^1(\bar{X};\bz)$, and the unitary-holomorphic decomposition 
$\bar{\xi}=\exp(\alpha)\exp(\beta)$ satisfies the hypothesis in Proposition~\ref{prop-iso0}.
The result follows for~$\bar{\xi}$ and we have achieved Step~\ref{step1}.

\begin{step}
Preparation of an induction process when $X$ is a quasiprojective surface.
\end{step}

Using the arguments of Step~\ref{step1}, we obtain
that either $\xi$ or $\bar \xi$ fits in a unitary-holomorphic decomposition
such that the Deligne extension of the unitary part is non-trivial.
We can assume that it is the case for $\xi=\psi\exp(\omega)$, $\psi$ unitary
with $\tilde{L}_\psi\not\cong\cO_{\bar{X}}$ (in particular, $\psi\neq\one$) and 
$\omega\in H^0(\bar{X};\Omega_{\bar{X}}^1(\log\cD^\xi))$.

We prove the statement by induction on the number $n$ of 
irreducible components of $\cD':=\cD-\cD^\xi$.
If $n=0$ the statement is trivially true. Let us assume $n>0$ and let us fix an irreducible
component $D$ of $\cD'$. Let us denote $\cD'':=\cD-D$ and $\xi''$ the character
induced by $\xi$ on $X'':=\bar{X}\setminus\cD''$. Let
$\check{D}:=D\setminus\cD''$.
Also set: 
$$\cL''=\Omega_{\bar{X}}^1(\log\cD'')\otimes \tilde{L}_\psi,\quad 
\cL=\Omega_{\bar{X}}^1(\log\cD)\otimes \tilde{L}_\psi.
$$
Since $H^0(\bar{X},\tilde{L}_\psi)$, we have that 
$$
\begin{matrix}
H^{\cO}_\xi=&\ker\left(H^0(\bar{X},\cL)
\overset{\wedge \omega}{\longrightarrow} 
H^0(\bar{X},\Omega_{\bar{X}}^2(\log\cD)\otimes \tilde{L}_\psi)\right),\hfill\\
H^{\cO}_{\xi''}=&\ker\left(H^0(\bar{X},\cL'')
\overset{\wedge \omega}{\longrightarrow} 
H^0(\bar{X},\Omega_{\bar{X}}^2(\log\cD'')\otimes \tilde{L}_\psi)\right),\hfill\\
H^{\overline{\cO}}_\xi=&\ker\left(H^1(\bar{X},\tilde{L}_\psi)\overset{\wedge \omega}{\longrightarrow} H^1(\bar{X},\cL)
\right),\hfill\\
H^{\overline{\cO}}_{\xi''}=&
\ker\left(H^1(\bar{X},\tilde{L}_\psi)\overset{\wedge \omega}{\longrightarrow} H^1(\bar{X},\cL'')
\right).\hfill
\end{matrix}
$$ 
Both $\cL$ and $\cL''$ fit in the following short exact sequence 
$$
0\to\cL''\to \cL\to
i_*(\tilde{L}_\psi)_{|D}\to 0
$$
and the associated long exact sequence
\begin{equation}
\label{eq-lex}
0\to H^0(\bar{X},\cL'')\to H^0(\bar{X},\cL)\overset{(\spadesuit)}{\to} H^0(D;(\tilde{L}_\psi)_{|D})
\overset{(*)}{\to} H^1(\bar{X},\cL'')\to H^1(\bar{X},\cL). 
\end{equation}

\begin{step}\label{step3}
If the map~$(*)$ of \eqref{eq-lex} vanishes then $H^{\overline{\cO}}_\xi=H^{\overline{\cO}}_{\xi''}$.
\end{step}
By the exactness of \eqref{eq-lex} the map $H^1(\bar{X},\cL'')\to H^1(\bar{X},\cL)$
is injective; since $\wedge\omega$ factorizes through this mapping in the definition of
$H^{\overline{\cO}}_\xi$, it is immediate that $H^{\overline{\cO}}_\xi=H^{\overline{\cO}}_{\xi''}$.
We have proved Step~\ref{step3}

\begin{step}\label{step4}
If the restriction of $\psi$ to $\check{D}$ is non-trivial then $H^{\overline{\cO}}_\xi=H^{\overline{\cO}}_{\xi''}$.
\end{step}
The third term $H^0(D;(\tilde{L}_\psi)_{|D})$ of \eqref{eq-lex} vanishes by Proposition~\ref{prop-neg}
and hence, also~$(*)$ does. The statement of Step~\ref{step3} implies Step~\ref{step4}. 

\begin{aspt}
According to Step~\ref{step4}, from now on we assume that the restriction of $\psi$ to $\check{D}$ is trivial. 
\end{aspt}

The character $\psi$ acts trivially on $\pi_1(\check{D})$. By Proposition~\ref{prop-neg}
$(\tilde{L}_\psi)_{|D}\cong\cO_D$.
Since $H^0(D;\cO_D)\cong\bc$, either $(\spadesuit)$ or $(*)$ vanishes, i.e.,
either $H^{\cO}_{\xi}\cong H^{\cO}_{\xi''}$ or
$H^{\overline{\cO}}_{\xi}= H^{\overline{\cO}}_{\xi''}$.

\begin{step}\label{step5}
If $\dim H^1(X;\bc_{\psi})=\dim H^1(X'';\bc_{\psi''})$ then $H^{\overline{\cO}}_\xi=H^{\overline{\cO}}_{\xi''}$.
\end{step}

It is immediate from the above property.

\begin{aspt}
According to Step~\ref{step5}, from now on we will assume $\dim H^1(X;\bc_{\xi})>\dim H^1(X'';\bc_{\xi''})$.
\end{aspt}
\begin{step}\label{step6}
$\dim H^1(X;\bc_{\psi})-\dim H^1(X'';\bc_{\psi''})=1$.
\end{step}

The above arguments imply $\dim H^1(X;\bc_{\xi})-\dim H^1(X'';\bc_{\xi''})=1$.
By duality 
$$
\dim H_1(X;\bc_{\xi^{-1}})-\dim H_1(X'';\bc_{(\xi'')^{-1}})=1.
$$
Using a Mayer-Vi{\'e}toris 
exact sequence for $H_1(X'';\bc_{\xi})$, we obtain that a meridian $\gamma$ around 
$D$ defines a twisted cycle $\zeta_D$ which determines a non-trivial homology class 
in the kernel of the natural map
$H_1(X;\bc_{\xi})\to H_1(X'';\bc_{\xi''})$ induced by inclusion.

More precisely, let $N$ be a regular neighborhood of $\check{D}$, and note
that $X''=X\cup N$. Then $X\cap N$ has the homotopy type of a Seifert $3$-manifold $M$, in fact
a circle bundle over $\check{D}$ and $N$ has the homotopy type of~$\check{D}$. 
The character is trivial on both $M$ and $\check{D}$; recall that $\dim H_1(M,\bc)\geq\dim H_1(\check{D};\bc)$
and equality arises only when $\check{D}$ is compact and its Euler number is non-zero.
Since the map $H_1(X,\bc_{\xi})\to H_1(X'',\bc_{\xi''})$ is surjective, the exact sequence
%
$$
H_1(M,\bc_{\xi})\to H_1(X,\bc_{\xi})\oplus H_1(\check{D},\bc_{\xi})\to
H_1(X'',\bc_{\xi''}) \to 0,
$$
can be completed to a short exact sequence and $\dim H_1(M,\bc)=\dim H_1(\check{D};\bc)+1$.
These arguments also hold for the characters $\xi_t:=\psi\exp(t\omega)$, $t\in\bc^*$.
Since the character $\psi$ is limit of the characters~$\xi_t$, we obtain that
the class of $\zeta_D$ is a non-trivial element of $H_1(X;\bc_{\psi})$; since the difference
of dimensions is at most one, Step~\ref{step6} is achieved.

%
%

\begin{step}\label{step7}
The map~$(\spadesuit)$ is non-trivial. 
\end{step}

As stated in Remark~\ref{rem-mhs}, for the character~$\psi$ the decomposition is natural
and we can evaluate the elements of $H^{\overline{\cO}}_{\psi}$ in $H_1(X;\bc_{\psi})$
and the elements of~$H^{\overline{\cO}}_{\psi''}$ in~$H_1(X'';\bc_{\psi''})$. Both spaces are equal to
$H^1(\bar{X},\tilde{L}_\psi)$.
Since $\zeta_D$ is a trivial cycle of $H_1(X'';\bc_{\psi''})$, $\forall\alpha\in H^1(\bar{X},\tilde{L}_\psi)$,
we have that $\alpha(\zeta_D)=0$. Using the duality between twisted homology and cohomology,
there exists $\beta\in H^{\cO}_{\psi}=H^0(D;(\tilde{L}_\psi)_{|D})$ such that 
$\beta(\gamma)\neq 0$. Then $\beta$ is an element of $H^0(\bar{X},\cL)\setminus H^0(\bar{X},\cL'')$.
The exactness of \eqref{eq-lex} implies the statement of Step~\ref{step7}.

We conclude that $(*)$ vanishes and hence  
$H^{\overline{\cO}}_{\xi}\cong H^{\overline{\cO}}_{\xi''}$. 
One can replace $X$ by $X''$ and apply the induction hypothesis once again.
\end{proof}


\section{Characters and orbifold maps}\label{sec-char}

The main tool for the proof of Theorem~\ref{thmprin} is the following result which corresponds to
\cite[Proposition~V.1.4]{ara:97}, except for this statement a weaker hypothesis is required 
(only non-torsion characters as opposed to non-unitary). In this section the notations used in 
\S\ref{sec-prelim} and \ref{sec-orb} will be followed, that is, $X$ is a smooth quasi-projective 
surface, $G:=\pi_1(X)$, $\bar{X}$ is a smooth projective compactification of $X$ such that 
$\cD:=\bar{X}\setminus X$ is a normal crossing divisor, and $C_\varphi$ 
is an orbifold coming from a Riemann surface $C$.

\begin{thm}\label{thm-pullback}
Let $\xi\in\bt_G$ be a non-torsion character in $\cV_1(X)\neq\emptyset$
and $0\neq\theta\in H^1(X;\bc_\xi)$. Then there exist:
\begin{itemize}
\item an orbifold $C_\varphi$,
\item an orbifold map $f:X\to C_\varphi$,
\item a character $\xi_C\in\cV_1(C_\varphi)$,
\item and a cohomology class $\theta_C\in H^1(C_{\varphi};\bc_{\xi_C})$
\end{itemize}
such that
$\xi=f^*(\xi_C)$ and $\theta=f^*(\theta_C)$.
\end{thm}



\begin{lem}\label{lema-pure}
It is enough to prove Theorem{\rm~\ref{thm-pullback}} for holomorphically pure elements 
in~$H^1(X;\bc_\xi)$.
\end{lem}

\begin{proof}
First note that if Theorem{\rm~\ref{thm-pullback}} is true for holomorphically pure elements 
in $H^1(X;\bc_\xi)$, then it is also true for anti-holomorphically pure elements by Theorem~\ref{prop-tm} 
and Proposition~\ref{prop-iso}.

Let us assume that the claim in Theorem{\rm~\ref{thm-pullback}} holds for holomorphically pure elements 
but not in general. Then there exist non-zero $\theta_1,\alpha_2\in  H^1(X;\bc_\xi)$, where $\theta_1$ is
holomorphically pure, and $\alpha_2$ is anti-holomorphically pure, coming from different orbifold morphisms, 
say $f_1:X\to C_1$ and $f_2:X\to C_2$ respectively. 
We may replace $C_2$ by $C_2\setminus\{p\}$ ($p\in C_2$) and $X$ by $f_2^{-1}(C_2\setminus\{p\})$.
By Corollary~\ref{cor-hol-1} and Proposition~\ref{prop-orbi-riemann} there exists a holomorphically 
pure element $\theta_2\neq 0$ which is a pull-back by~$f_2$.

Since $\theta_1$ is a twisted logarithmic $1$-form, $p\mapsto\ker(\theta_1)_p=\ker (df_1)_p$ 
defines a foliation
$\cF_1$ which determines and is determined by $f_1$.
Analogously, we construct a foliation $\cF_2\neq\cF_1$. 
Choose a point $q\in X$ such that $(\cF_2)_q\neq(\cF_1)_q$.
Let $(t_1,t_2)\in\bc^2\setminus\{(0,0)\}$ and consider
the holomorphically pure element $\theta_{(t_1,t_2)}:=t_1\theta_1+t_2\theta_2$.
Since we claim that Theorem{\rm~\ref{thm-pullback}} holds for holomorphically pure elements
and two proportional elements come from the same orbifold morphism, for any $[t_1:t_2]\in\bp^1$
we obtain an orbifold morphism $f_{[t_1:t_2]}:X\to C_{[t_1:t_2]}$ such that
for generic $p\in X$ we have $\ker(\theta_{(t_1,t_2)})_p=\ker (df_{[t_1:t_2]})_p$.
Since these morphisms have distinct fibers we have obtained a
family of pairwise non-equivalent orbifold morphisms onto an hyperbolic orbifold parametrized
by $\bp^1$. The set of equivalence classes of such morphism is at most countable,
see~\cite[Lemma~V.1.5]{ara:97} and we obtain 
a contradiction. 

\end{proof}

The following key result is the orbifold version of~\cite[Proposition~V.1.4]{ara:97}, which in turn, 
is the quasi-projective version of~\cite[Proposition~2.1]{be}. We adapt Arapura's 
proofs~\cite[Propositions~V.1.3,~V.1.4]{ara:97} stressing the details required for the orbifold version.

\begin{prop}\label{prop-ara-beau}
Let $\xi=\psi\exp(\omega)\in \bt_G$ be a strict unitary-holomorphic decomposition of the character $\xi$
and let $\theta\in H^1(X;\bc_\xi)$ be a holomorphically pure element. 
Then, there exist:
\begin{enumerate}
\enet{\rm(\arabic{enumi})} 
\item an orbifold map $f:X\to C_\varphi$,
\item $\xi_C\in\bt_\Pi$, $\Pi:=\pi_1^{\text{\rm orb}}(C_\varphi)$, and
\item $\theta_C\in H^1(C_\varphi;\bc_{\xi_C})$
\end{enumerate}
such that $\xi=f^*(\xi_C)$ and $\theta=f^*(\theta_C)$.
\end{prop}

\begin{proof} 
Recall that the Deligne extension $\tilde{L}_\psi$ is an extension
for $L_\xi$ with a meromorphic connection $\nabla_\omega$. Using \eqref{spseq}
$\theta$ can be represented by a section
$\eta\in H^0(\bar{X};\Omega_{\bar{X}}^1(\log\cD)\otimes\tilde{L}_\psi)$
such that $\nabla_\omega(\eta)=0$ and it is not the image by $\nabla_\omega$ of a 
holomorphic section of $\tilde{L}_\psi$. 

According to Theorem~\ref{prop-dlg}, $\eta$ is so that $\eta\wedge\omega=0$. In addition, if 
$\tilde{L}_\psi=\cO_{\bar{X}}$, then $\eta$ is not a complex multiple of $\omega$. 

Using~\cite[Proposition~V.1.3]{ara:97}, a holomorphic mapping $f:X\to C$
onto a quasi-projective Riemann surface $C$ exists with the following properties:

\begin{enumerate}
\item 
The mapping is the restriction of $\bar{f}:\bar{X}\to\bar{C}$ (possibly after 
performing additional blowing-ups on $\bar{X}$), $\cD_1:=\bar{C}\setminus C$, 
\item
there is a logarithmic $1$-form $\omega_1\in H^0(\bar{C};\Omega_{\bar{C}}^1(\log\cD_1))$ 
such that $\omega=\bar{f}^*(\omega_C)$.
\end{enumerate}

At this point it is worth mentioning that~\cite[Proposition~V.1.3]{ara:97} also ensures
the existence of a character on $C$, whose pull-back by $f$ translated by a torsion element,
equals $\psi$. Alternatively, we will use a more detailed description of $\psi$ to describe it 
as the pull-back of a character on a certain orbifold structure on $C$.

Let $\varphi$ be the maximal orbifold structure on $C$ naturally induced by $f$ as described
in Remark~\ref{rem-maximal}, and let $\check{C}$ be the set 
of non-multiple points for $\varphi$. Let us write $\bar{C}\setminus\check{C}=\cD_1+\cD_2$,
$\check{X}:=f^{-1}(\check{C})$, and $\check{f}:=f_{|}:\check{X}\to\check{C}$.
Let $\check{\psi}$ denote the induced character by $\psi$ on~$\check X$. 


Let us consider $F$ a generic fiber of $\check{f}$ and consider its closure $\bar{F}$ 
in $\bar{X}$. Outside the multiple fibers there is an exact sequence (see~\cite[Lemma 1.5C]{nr:83} and~\cite{cko})
\[
\pi_1(F)\to\pi_1(\check{X})\to\pi_1(\check{C})\to 1. 
\]
In order to check that $\check{\psi}$ is the pullback of a character $\psi_{\check{C}}$ of 
$\pi_1(\check{C})$ it is enough to check that $\pi_1(F)\subset\ker\check{\psi}$. Arapura proceeds 
as follows: there is a meromorphic section $\beta$ of $\tilde{L}_\psi$ such that $\eta=\omega\otimes\beta$.
Since $F$ is generic, one may assume $\omega_{|\bar{F}}\neq 0$. Hence, $\beta_{|\bar{F}}$ 
is holomorphic. This fact has two consequences according to Proposition~\ref{prop-neg} applied to~$F$:
\begin{itemize}
\item $(\tilde{L}_\psi)_{|\bar{F}}=\cO_{\bar F}$ and hence $\psi$ and $\check{\psi}$ are 
trivial on $F$, thus $\check{\psi}$ is a pullback, say $\check \psi=\check f^*(\psi_{\check C})$.
It follows that $\check{\xi}:=\check{\psi}\exp(\omega)$ is also a pullback, say
$\check \xi=\check f^*(\xi_{\check C})$, where $\xi_{\check C}=\psi_{\check C}\exp(\omega_{\check C})$.
\item The meromorphic section~$\beta$ is holomorphic and constant on~$\bar{F}$, i.e.,
$\beta$ is a pullback by $\bar{f}$, say $\beta=\bar{f}^*(\beta_{\bar C})$. 
Hence  $\eta$ is the pullback by $\bar{f}$ of a logarithmic form, that is 
$\eta=\bar f^*(\omega_{\bar C}\otimes \beta_{\bar C})$ (with poles along $\cD_1+\cD_2$).
\item Moreover, the character $\psi_{\check C}$ induces a character $\psi_{C_\varphi}$ on $C_\varphi$.
In order to see this note that, if $f^*(p)=\sum n_iE_i$, $\mu_i$ is a meridian around $E_i$, and 
$\mu_p$ is a meridian around $p$, then
$1=\psi(\mu_i)=\check \psi(\mu_i)=\psi_{\check C}(\mu_p^{n_i})=\left(\psi_{\check C}(\mu_p)\right)^{n_i}$.
Therefore $\left(\psi_{\check C}(\mu_p)\right)^{\varphi(p)}=1$, since 
by construction $\varphi(p)=\gcd(n_i)$.
\end{itemize}

The existence of $\check C$, $\xi_{\check C}$, and $\theta_{\check C}$ proves the result for 
$\check{X}$ with respect to a Riemann surface. Finally, the discussion above and 
Proposition~\ref{prop-orbi-riemann} show the existence of $C_\varphi$, $\xi_{C}$, and $\theta_{C}$ 
which proves the result for $X$ with respect to orbifolds.
\end{proof}

The non-unitary case for Theorem~\ref{thm-pullback} is partially proved 
in~\cite[Proposition~V.1.4]{ara:97}. In fact, the argumentation line in the proof only
establishes the statement for characters that ramify over all components of the divisor 
$\cD=\bar{X}\setminus X$. 
More precisely, since Theorem~\ref{prop-tm} is needed, Arapura 
replaces the quasi-projective manifold~$X$ by the bigger manifold $X^\xi\subset \bar{X}$ but, 
as Examples~\ref{ex-notram} and~\ref{ex-elliptic} show, this must be done carefully. 
Proposition~\ref{prop-iso} solves these issues.


\begin{cor}\label{cor-caso1}
Theorem{\rm~\ref{thm-pullback}} is true when the character~$\xi$ is either non-unitary
or $b_1(X^\xi)>b_1(\bar{X})$.
\end{cor}

\begin{proof}
It is a direct consequence of Proposition~\ref{prop-ara-beau} and Lemmas~\ref{lema-pure} and~\ref{lema-char}.
\end{proof}

\begin{rem}
This corollary is enough if we restrict our attention to the case when $b_1(\bar{X})=0$ as in~\cite{li:09}. 
\end{rem}

The remaining case, that is, non-torsion unitary characters~$\xi$ such that 
$b_1(X^\xi)=b_1(\bar{X})$, will be treated with different arguments, following Delzant's
technique in~\cite{delzant} for K\"ahler manifolds. This strategy uses results by 
Levitt~\cite{lev:94} and Simpson~\cite{Si1} which can be adapted to our case.

\begin{prop}\label{prop-caso2}
Let $\xi$ be a non-torsion unitary character such that $b_1(X^\xi)=b_1(\bar{X})$ and 
let $0\neq\theta\in H^1(X;\bc_\xi)$. Then, Theorem{\rm~\ref{thm-pullback}} holds.
\end{prop}

\begin{proof}
Let us assume that $\xi$ is not algebraic. Since $\cV_k:=\cV_k(X)$ is an algebraic 
subvariety of $\bt_G$, $G=\pi_1(X)$, defined with rational coefficients, one can deduce 
that $\xi$ cannot be an isolated point of $\cV_k$. Hence, if $k=\dim H^1(X;\bc_\xi)$ and 
$V$ is the irreducible component of $\cV_k$ containing $\xi$, this component contains 
non-unitary elements and the result follows easily.

From now on, $\xi$ will be assumed to be algebraic. If $\xi(G)$ only contains algebraic 
integers, two cases may happen: either all the conjugates of these algebraic integers are 
also algebraic integers, or not. In the first case Kronecker Theorem implies that $\xi$ is 
torsion and this case has been excluded. In the second case, a conjugate character 
$\tilde{\xi}$ is non-integral. Since all the statements are algebraic, it is enough to 
prove the result for~$\tilde{\xi}$.

Then, we can assume $\xi$ is algebraic and non-integral and we can follow the proof 
of~\cite[Proposition~2]{delzant}. The first step is to construct an
\emph{exceptional} class $\omega\in H^1(X;\br)\setminus\{0\}$ in the sense
of Bieri-Neumann-Strebel~\cite{bns}; in fact, $\omega$ is defined 
with integer coefficients. We identify $H^1(X;\br)$ with
$\Hom(G,\br)$. 

We sketch the construction of~$\omega$, see~\cite{delzant} for more details.
Since $H^1(X;\bc_\xi)$ is generated by elements in a number field we may assume that $\theta$ 
is represented by a cocycle (see~\ref{alg-const}) in a number field (also denoted by~$\theta$). 
We fix such a number field~$\bk$ in order to have that $\forall g\in G$ the matrix 
$\left(\begin{smallmatrix}
\xi(g)&\theta(g)\\
0&1
\end{smallmatrix}\right)$
has coefficients in~$\bk$. This defines an action of $G$ on $\bk^2$ and $\bp^1(\bk)$. 
The fact that $\xi$ is not integral guarantees the existence of a valuation~$\nu$ such 
that $\omega:=\nu\circ\xi$ is not trivial. This element is the Busemann cocycle of an 
exceptional action of $G$ on the Bruhat-Tits tree $T_\nu$ (see~\cite{bt:72}). As a 
consequence, the class~$\omega$ is exceptional. 

Since $\omega:=\nu\circ\xi$, 
we deduce that $\ker\xi\subset\ker\omega$. As a consequence, $\omega$ is the restriction of 
$0\neq\omega_\xi\in H^1(X^\xi;\br)$. We can strengthen this argument;
since the target of $\nu$ is an ordered group, without torsion. Thus for any $g\in G$ whose image
$\xi(g)$ is a torsion element, one has that $g\in\ker\omega$. Hence, let 
$\mu$ be the meridian of an irreducible component
of $\cD^\xi$. Since $b_1(X^\xi)=b_1(\bar{X})$, $\xi(\mu)$ is a torsion element and,
applying last comment, $\mu\in\ker\omega$. Hence, there exists 
$\omega_{\bar{X}}\in H^1(\bar{X};\br)\setminus\{0\}$ such that 
$\omega_\xi$ is the restriction of $\omega_{\bar{X}}$ (in the same way,
$\omega$ is the restriction of $\omega_{\bar{X}}$). We represent
$\omega_{\bar{X}}$ by a closed differential $1$-form (and $\omega$ by its restriction to $X$).
For the sake of simplicity we keep the notation of $\omega_{\bar{X}}$ and $\omega$
for the differential $1$-forms.

Let us consider an unramified abelian 
covering $\bar{\pi}:\bar{X}_\omega\to\bar{X}$ such that $\bar{\pi}^*\omega_{\bar{X}}$
is an exact $1$-form (e.g, the one determined by $\ker\omega_{\bar{X}}$ as a character of
$\pi_1(\bar{X})$); let $\bar{F}:X_\omega\to\br$ a primitive of $\bar{\pi}^*\omega_{\bar{X}}$. 
Let $X_\omega:=\bar{\pi}^{-1}(X)$, since $\pi_1(X)\to\pi_1(\bar{X})$ is surjective
the manifold  $X_\omega$ is connected. Let us denote $\pi:X_\omega\to X$ the induced covering
and $F:X_\omega\to\br$ the restriction of $\bar{F}$; $F$ is a primitive of $\pi^*\omega$.

Following~\cite{lev:94}, we have a characterization of exceptional elements of $H^1(X;\br)$ 
when they are represented by~$1$-forms:  $\omega$ is exceptional if and only 
if $F^{-1}(\br_{>0})$ has more than one connected component where $F$ is unbounded. Since 
$\bar{X}\setminus X$ is a union of real codimension $2$ varieties
the number of connected components of $F^{-1}(\br_{>0})$ and $\bar{F}^{-1}(\br_{>0})$ is the same.
We conclude that $\omega_{\bar{X}}$ is also exceptional.

Using Simpson's results~\cite{Si2}, one obtains a surjective morphism 
$\bar{f}:\bar{X}\to\bar{C}$ with connected fibers, where $\bar{C}$ is a compact Riemann surface. 
Let $C:=\bar{f}(\bar{C})$ and let $f:=\bar{f}_{|}:X\to C$ be the restriction morphism.
Here $C$ is considered with the maximal orbifold structure $\varphi$ defined by~$f$ (see Remark~\ref{rem-maximal}).
We follow again the arguments of the proof of~\cite[Proposition~2]{delzant}. For this orbifold
$C_\varphi$, the character $\xi$ is in the pull-back of the characteristic variety $\Sigma(C_\varphi)$. 
This is proved using the action on the tree and showing that the action on the projective line~$\bp^1(\bk)$ 
is trivial. As a consequence $\theta$ is in the pull-back of $\Sigma(C_\varphi)$.
\end{proof}

\begin{proof}[Proof of Theorem{\rm~\ref{thm-pullback}}]
Let $\xi\in \bt_G$ be a character on $X$. If $\xi$ is not unitary, then Corollary~\ref{cor-caso1} gives
the result. Hence we might assume that $\xi$ is a unitary character.

On the other hand, if $b_1(X^\xi)>b_1(\bar X)$, then again Corollary~\ref{cor-caso1} gives the result.
Finally, if $b_1(X^\xi)=b_1(\bar X)$, then Proposition~\ref{prop-caso2} shows the statement.
\end{proof}

\begin{rem}
\label{rem-4.1}
Note that the proof of Theorem{\rm~\ref{thm-pullback}} in fact shows a sharper result, namely, it is
also true for unitary characters (torsion or non-torsion) such that $b_1(X^\xi)>b_1(\bar X)$. In other
words, the only case missing is unitary torsion characters for which $b_1(X^\xi)=b_1(\bar X)$. In fact,
the statement is not true in general for this case as illustrated in Example~\ref{ex-aarhus}.
\end{rem}

\begin{proof}[Proof of Theorem{\rm~\ref{thmprin}}]
Once Theorem~\ref{thm-pullback} is proved, the arguments of~\cite[Section~V]{ara:97} 
give the result (they formally apply to $\cV_k$ for any~$k$).
\end{proof}

\section{Behavior of torsion characters and further applications}\label{sec-torsion}

The results of \S~\ref{sec-char} do not apply to general torsion characters
of a quasi-projective manifold~$X$. Some of them may appear as isolated points of 
some characteristic variety $\cV_k(X)$ and in that case they may fall either in 
case~\ref{thmprin-orb} or in case~\ref{thmprin-tors} of Theorem~\ref{thmprin}. Since 
the irreducible components of Theorem~\ref{thmprin}\ref{thmprin-orb} are torsion-translated 
subtori of $\bt_G$, $G:=\pi_1(X)$, even if Theorem~\ref{thm-pullback} does not apply
to torsion characters, the properties of \emph{close} non-torsion points imply
that some of the elements of $H^1(X;\bc_\xi)$, for $\xi$ torsion, do come from an orbifold map.
In this section we study the behavior of torsion characters as well as show some properties of 
characteristic varieties which can be derived from Theorem~\ref{thmprin}. 

\begin{prop}
Let $\xi\in\bt_G$ be a torsion character such that $b_1(X^\xi)>b_1(\bar{X})$.
Then $\xi$ cannot be an isolated component of any characteristic variety of~$X$.
\end{prop}

\begin{proof}
It is a direct consequence of Lemma~\ref{lema-char} and Corollary~\ref{cor-caso1}:
there is a strict unitary-holomorphic decomposition $\xi=\psi\exp(\omega)$.
Let us assume that $\xi\in\cV_k(X)$.
Following Remark~\ref{rem-nabla} and the ideas in~\cite{be} and in~\cite[Proposition~V.1.3]{ara:97},
for any $t\in\bc$, we have $\xi_t:=\xi=\psi\exp(t\omega)\in\cV_k(X)$, and then
$\xi$ cannot be an isolated point of~$\cV_k(X)$.
\end{proof}

Next we give a generalization of the results about essential coordinate components 
of the characteristic varieties of the complement of hypersurfaces in $\bp^n$, 
see~\cite{acc:pcps,li:01}. Let $X$ be a quasi-projective manifold and let $V$ be an irreducible 
component of $\cV_k(X)$. Let $X^V$ be the maximal subvariety $X\subset X^V\subset\bar{X}$
such that any character in $V$ is defined over $X^V$. 
Recall that $H^1(X^V;\bc^*)\subset H^1(X;\bc^*)$ and that
$V\subset H^1(X^V;\bc^*)$. The following result
is a generalization of a Libgober's result in~\cite[Lemma~1.4.3]{li:01}.

\begin{prop}
If $V$ is not contained in $\cV_k(X^V)$ then $V$ is a torsion
point of type \ref{thmprin-tors} in Theorem{\rm~\ref{thmprin}}.
\end{prop}

\begin{proof}
If $V$ is not isolated then it comes from an orbifold map
$X\to C_\varphi$. Following the ideas in~\cite[Lemma~1.4.3]{li:01}
we can extend this map to $X^V$ using the definition of this variety.
\end{proof}

We recall a definition introduced in~\cite{ACM-artin,DPS4}.

\begin{dfn}
For $V$ an irreducible component of $\cV_k(G)$ such that $\dim_\bc V\geq 1$, consider
$\Sh V$ (not necessarily in $\cV_k(G)$) parallel to $V$ ($\Sh V=\rho V$ for some 
$\rho\in \bt_G$) and such that $\one \in \Sh V$. Such a subtorus $\Sh V$ will be 
referred to as the \emph{shadow} of~$V$.
\end{dfn}

The result and the proof of Theorem~\ref{thmprin} provide many obstructions 
for the quasi-projectivity of a group. We start with a very useful result
(see~Lemma~\ref{lema-pure}).

\begin{lem}\label{lema-unico}
Let $\xi$ be a non-torsion character such that $H^1(X;\bc_\xi)\neq 0$. Then there is a 
unique \emph{maximal} orbifold mapping $f:X\to C_\varphi$ such that $\xi=f^*\xi_C$ 
and $H^1(X;\bc_\xi)=f^*H^1(C_\varphi;\bc_{\xi_C})$.
\end{lem}

\begin{proof}
Exchanging $\xi$ by $\bar{\xi}$ if necessary we may assume that $H_\xi^{\cO}\neq 0$.
Let us assume that $f_j:X\to C_{j,\varphi}$, $j=1,2$, satisfies $\xi=f^*\xi_{C_j}$ and
$f_j^*H^1(C_{j,\varphi};\bc_{\xi_{C_j}})\subset H^1(X;\bc_\xi)$.

One can remove points from $C_j$ to ensure $H_{\bar{\xi}_j}^{\cO}\neq 0$ (replacing $X$ 
by a dense Zariski-open set). The foliation argument used in Lemma~\ref{lema-pure} gives 
a contradiction.
\end{proof}

\begin{prop}\label{prop-obs1}
Let $G$ be a quasi-projective group and $V_1,V_2$ different irreducible components of 
$\cV_k(G)$, $k\geq 1$ of positive dimension. Then
\begin{enumerate}
\enet{\rm(\arabic{enumi})}
\item\label{prop-obs1-1} If the intersection $V_1\cap V_2$ is non-empty, then it consists 
of isolated torsion points. 
\item\label{prop-obs1-2} 
Their shadows are either equal or have $\one$ as an isolated intersection point.
\item\label{prop-obs1-3} 
If $V_1$ is not a component of $\cV_{k+1}(G)$ and $p\in V_1\cap\cV_{k+1}(G)$, then 
$p$~is a torsion point.
\item\label{prop-obs1-4} $V_1$ is an irreducible component of $\cV_\ell(G)$, $1\leq\ell\leq k$. 
\end{enumerate}
\end{prop}

\begin{proof}
The property~\ref{prop-obs1-4} follows from Lemma~\ref{lema-unico} since it is true for 
orbifolds. Each irreducible component comes from an orbifold; if two different components 
come from the same orbifold, they are parallel, hence disjoint. Lemma~\ref{lema-unico} 
gives~\ref{prop-obs1-1}.

If the underlying Riemann surface of the orbifold giving $V_j$ is not $\bc^*$ or an 
elliptic curve, then their shadows are in $\cV_1$, and \ref{prop-obs1-2} follows 
from~\ref{prop-obs1-1}. For the general case, it is proved by Dimca. Finally~\ref{prop-obs1-3} 
is an immediate consequence of \ref{prop-obs1-1} and Theorem~\ref{thmprin}.
\end{proof}

\begin{rem}
Parts~\ref{prop-obs1-1} and~\ref{prop-obs1-2} in Proposition~\ref{prop-obs1} can also 
be found in~\cite{DPS4}, Part~\ref{prop-obs1-3} is proved in~\cite{Di3}.
\end{rem}

\begin{prop}\label{prop-obs2}
Let $G$ be a quasi-projective group and let $V$ be an irreducible component of $\cV_k(G)$, 
$k\geq 1$ of positive dimension~$d$. Then:
\begin{enumerate}
\enet{\rm(\arabic{enumi})}
\item \label{prop-obs2-1} 
If $\one\in V $, then $k\leq d-1$. Moreover, one can ensure that $V$ is a component of 
$\cV_{d-e}(G)$, where $e=2$ if $d$ even and $e=1$ if $d$ odd.
\item\label{prop-obs2-2} 
If $\one\notin V $, then $ V $ is a component of~$\cV_d(G)$. 
\item\label{prop-obs2-3} 
If $\one\notin V $ and $d>2$, then its shadow is an irreducible component of~$\cV_1(G)$.
\item\label{prop-obs2-5} 
If $\one\notin V $ and $d=2$, then its shadow
is an irreducible component of $\cV_1(G)$ if and only if its shadow is an irreducible component 
of~$\cV_2(G)$.
\item\label{prop-obs2-6} 
If $\one\notin V $ and $d=1$, then its shadow is not an irreducible component of~$\cV_1(G)$.
\end{enumerate}
\end{prop}

\begin{proof}
It is a direct consequence of Theorem~\ref{thmprin} and the properties of the characteristic 
varieties of orbifolds.
\end{proof}

\begin{rem}
The results \ref{prop-obs2-5} and~\ref{prop-obs2-6} in Proposition~\ref{prop-obs2} can 
be found in~\cite{Di3}. The cases where the shadow is not in the characteristic variety
correspond, according to Theorem~\ref{thmprin}, to either orbifold pencils over $\bc^*$ 
or elliptic pencils.
\end{rem}

\begin{prop}
\label{prop-suma}
Let $G$ be a quasi-projective group, and let $V_1$ and $V_2$ be two distinct irreducible components 
of $\cV_k(G)$, resp. $\cV_\ell(G)$. If $\xi\in V_1\cap V_2$, then this torsion point satisfies $\xi\in\cV_{k+\ell}(G)$.
\end{prop}

\begin{proof}
Let $H_j\subset H^1(X;\bc_\xi)$ be the subspace obtained by the pull-back of the orbifold
giving $V_j$, $j=1,2$, $\dim H_1=k$, $\dim H_2=\ell$.
Using the arguments of the proof of Lemma~\ref{lema-unico} we prove $H_1\cap H_2=\{0\}$.
\end{proof}

\begin{rem}
\label{rem-k+l}
A careful look at the proof of Proposition~\ref{prop-suma} shows stronger consequences.
In particular, the hypothesis that the irreducible components $V_1$ and $V_2$ be \emph{distinct}
can be weakened as follows: $V_1$ and $V_2$ are \emph{$\xi$-distinct} as long as the spaces 
$H_j\subset H^1(X;\bc_\xi)$ obtained in the proof are different. Analogously, using
the arguments of the proof of Lemma~\ref{lema-unico} one obtains that $H_1\cap H_2=\{0\}$
and hence~$\xi\in\cV_{k+\ell}(G)$.

This subtle improvement is illustrated in Example~\ref{ex-c6-9k}.
\end{rem}

\section{Examples}
\label{sec-exam}

\begin{example}
Consider the curve $C$ with equation
$$
x y z\left(x^2+y^2+z^2-2 x y -2 x z-2 y z\right)=0.
$$
The fundamental group~$G$ of $X:=\bp^2\setminus C$ is the Artin group
of the triangle with weights~$(2,4,4)$, see~\cite{ACM-artin}. The first characteristic
variety of~$X$ consists of three irreducible components of dimension~$1$,
intersecting at one character~$\xi$ of order~$2$. The second characteristic variety
equals~$\{\xi\}$ and the third one is empty. This illustrates Proposition~\ref{prop-suma}.
\end{example}

\begin{example}
\label{ex-c6-9k}
Let us consider a curve $C\subset\bp^2$ which is the dual curve of a smooth cubic and
let $X:=\bp^2\setminus C$. Note that $C$ is a curve of degree~$6$ with nine ordinary cusps
and $\bt_G$ is the set of sixth roots of unity, where $G:=\pi_1(X)$.
The fundamental group of this curve $G$ was computed by Zariski~\cite{zr:37}.
It is not difficult to prove that $\cV_k(X)$, $k=1,2,3$, equals 
the set~$\{\xi_6,\bar \xi_6\}$ of points of order~$6$, 
$\xi_6:=\exp\left(\frac{\sqrt{-1}\pi}{3}\right)$.
Concretely, let us define a curve~$C$ of equation:
$$
f(x,y,z):=x^6+y^6+z^6-2(x^3 y^3+ x^3 z^3+ y^3 z^3)=0.
$$
Since 

\begin{equation}
\label{eq-23-1}
f(x,y,z)=(x^3+y^3-z^3)^2-4 (x y)^3,
\end{equation}
the map $[x:y:z]\mapsto [(x^3+y^3-z^3)^2:(x y)^3]$ induces an orbifold rational map 
$\bp^2\dasharrow \bp^1_{2,3}$, for which $C\mapsto [4:1]$. Therefore it also determines an 
orbifold map $X\to\bc_{2,3}$ which produces (by pull-back) a non-zero element 
$\theta_z\in H^\cO_{\xi_6}\subset H^1(X;\bc_{\xi_6})$. From the symmetry of the equation one can
obtain two other elements $\theta_y$ and $\theta_x$. Even though $\theta_x$, $\theta_y$, and 
$\theta_z$ satisfy the statement of Theorem~\ref{thm-pullback}, one can easily construct 
1-forms not satisfying it by simply considering a generic linear combination of $\theta_z$ 
and $\theta_y$. 

Also note that one can find other toric decompositions, like the one defined by
\begin{equation}
\label{eq-23-2}
\begin{split}
f(x,y,z)&=4\left({x}^{2}+xz+xy+{z}^{2}+zy+{y}^{2}\right)^3\\
-&3\left({z}^{3}+2\,x{z}^{2}+2\,{x}^{2}z+{x}^{3}+2\,{z}^{2}y+2\,zyx+2\,{x}^{2}y
+2\,z{y}^{2}+2\,x{y}^{2}+{y}^{3}
\right)^2.
\end{split}
\end{equation}
In this case, the forms obtained using these pencils generate $H^1(X;\bc_{\xi_6})$.
This example can also be found in~\cite{CL-prep}, where an infinite number of decompositions
of type $fh^6=g_3^2+g_3^2$ are described. Such decompositions are called \emph{quasi-toric}
decompositions and they produce morphisms onto $\bp^1_{(2,3,6)}$.

This illustrates Remark~\ref{rem-k+l}, since for instance 
$V_1=V_2=\{\xi_6\}\subset \Sigma_1(X)$ can be chosen such that $V_1$ (resp. $V_2$) is associated 
with the orbifold map coming from~\eqref{eq-23-1} (resp.~\eqref{eq-23-2}). This automatically 
implies that $\xi_6\in \Sigma_2(X)$. 

Note that $\bt_\Pi=\cup_{\lambda} \bt_\Pi^{\lambda}$, where $\Pi=\pi^{\text{\rm orb}}(\bc_{2,3})=\bff^0_{2,3}$,
$\lambda=(\lambda_1,\lambda_2)\in C_2\times C_3$, and $\bt_\Pi^{(1,1)}=\{\one\}$. According
to Proposition~\ref{charvar-libre}, 
$\Sigma_1(\Pi)=\bt_\Pi^{(-1,\xi_3)} \cup \bt_\Pi^{(-1,\bar \xi_3)}=\{\xi_6, \bar \xi_6\}$ and
$\Sigma_2(\Pi)=\emptyset$ ($\xi_n:=\exp\left(\frac{2\pi \sqrt{-1}}{n}\right)$).

Summarizing, $\xi_6$ as an element of $\Sigma_1$ is of type~\ref{thmprin-orb}. 
However, if we consider $\xi_6$ as an element of $\Sigma_2$, then it is of type~\ref{thmprin-tors}.
\end{example}

\begin{example}
\label{ex-aarhus}
The affine Degtyarev curve $C\subset \bc^2$ in~\cite[Section~3]{AC-prep} provides an example of a space 
$X=\bc^2\setminus C$ such that $\Sigma_1(X)=\{\xi \mid \xi^4-\xi^3+\xi^2-\xi+1=0\} \cup \{\one\}$,
but $\xi$ is of type~\ref{thmprin-tors} in~Theorem~\ref{thmprin}. Note that, in this case 
$\xi$ is a torsion character and both $b_1(X^{\xi})=b_1(\bar X)=0$, which is necessary 
by Remark~\ref{rem-4.1}.
\end{example}

\begin{example}
Consider the following arrangement 
(cf.~\cite{hirzebruch-geraden}) given by equations 
$C_n:=(y^n-x^n)(y^n-z^n)(z^n-x^n)$, $n\geq 2$, where
$$
\begin{matrix}
\ell_{1+3k}:=y-\xi_n^k z, &
\ell_{2+3k}:=z-\xi_n^k x, &
\ell_{3+3k}:=y-\xi_n^k x.
\end{matrix}
$$ 
This arrangement can be seen as the Kummer covering 
$[x:y:z]\mapsto [x^n:y^n:z^n]$ of a Ceva arrangement ramified 
along $\{xyz=0\}$. It has been considered in~\cite{cohen-torsion,cohen-triples,Di4}
where mainly the components in $\Sigma_1$ have been accounted for. 
Here we interpret the essential components in $\Sigma_k$ for $k>1$.

Note the following pencils associated with the arrangement $C_n$:
$$
F_{\alpha,\beta}=\alpha f_1+\beta f_2:=
\alpha x^n \prod_{k=1}^n \ell_{1+3k}+ 
\beta y^n \prod_{k=1}^n \ell_{2+3k} = 
\alpha x^n(y^n-z^n)+\beta y^n(z^n-x^n),
$$
where 
$$
f_1+f_2=z^n(y^n-x^n)=
z^n \prod_{k=1}^n \ell_{3+3k},
$$
and
$$
G_{\alpha,\beta}=
\alpha g_1+\beta g_2:=
\alpha \prod_{k=1}^n \ell_{1+3k}+ 
\beta \prod_{k=1}^n \ell_{2+3k}
=\alpha (y^n-z^n)+\beta (z^n-x^n),
$$
where
$$
g_1+g_2=(y^n-x^n)=
\prod_{k=1}^n \ell_{3+3k}.
$$
The rational maps
$f_i:\mathbb P^2 \dasharrow \mathbb P^1$, $i=1,2$ defined by
$f_1([x:y:z]):=[f_1:f_2]$, $f_2([x:y:z]):=[g_1:g_2]$, 
are such that 
$\{xyz=0\}\cup C_n=f_1^{-1}(\{[0:1],[1:0],[1:-1]\})$
and
$C_n=f_2^{-1}(\{[0:1],[1:0],[1:-1]\})$.
Moreover, after resolving the base points of the rational
maps one obtains two morphisms:
$$\tilde f_1:X_n\to \Omega_1:=\mathbb P^1_{n,n,n},$$
$$\tilde f_2:X_n\to 
\Omega_2:=\mathbb P^1\setminus \{[0:1],[1:0],[1:-1]\},$$
where $X_n=\mathbb P^2\setminus C_n$.
By Proposition~\ref{charvar-compacto}, 
$\Sigma_1(\Omega_1)=\{\lambda=(\xi_n^i,\xi_n^j,\xi_n^{-(i+j)})\mid
\ell(\lambda)\geq 3\}\supset \Sigma_2(\Omega_1)=\emptyset$,
$\Sigma_1(\Omega_2)=\mathbb T_{\mathbb F^2}
\supset \Sigma_2(\Omega_2)=\emptyset$.
The injection 
$\tilde f_1^*(\Sigma_1(\Omega_1))\subset \Sigma_1(X_n)$
produces a zero-dimensional embedded component inside the one-dimensional
component $\tilde f_2^*(\Sigma_2(\Omega_2))\subset \Sigma_2(X_n)$
(except for $n=2$). Using Proposition~\ref{prop-suma} one can deduce that 
in fact $\tilde f_1^*(\Sigma_1(\Omega_1))\subset \Sigma_2(X)$. This was 
already pointed out in~\cite{Di4}, but here we give a different approach.
\end{example}


\providecommand{\bysame}{\leavevmode\hbox to3em{\hrulefill}\thinspace}
\providecommand{\MR}{\relax\ifhmode\unskip\space\fi MR }
\providecommand{\MRhref}[2]{%
  \href{http://www.ams.org/mathscinet-getitem?mr=#1}{#2}
}
\providecommand{\href}[2]{#2}

\end{document}